\documentclass{article}

\usepackage{tikz-cd}
\usepackage{amsmath,amsthm,amssymb}
\usepackage{verbatim}
\usepackage{authblk}
\usepackage{thmtools, thm-restate}

\declaretheorem{theorem}
\declaretheorem[name=Definition]{definition}
\declaretheorem[name=Proposition]{proposition}
\declaretheorem[name=Lemma]{lemma}
\declaretheorem[name=Corollary]{corollary}
\declaretheorem[name=Example]{example}
\declaretheorem[name=Remark]{remark}

\newcommand{\ma}{\textrm{MA}}

\newcommand{\supp}{\textrm{supp}}
\newcommand{\MA}{\ma}

\newcommand{\SL}{\textrm{SL}}
\newcommand{\R}{\mathbb{R}}
\newcommand{\vol}{\textrm{Vol}}
\newcommand{\id}{\textrm{id}}
\newcommand{\Ric}{\textrm{Ric}}

\title{An optimal transport approach to Monge-Amp\`ere equations on compact Hessian manifolds}

\author{Jakob Hultgren, Magnus \"{O}nnheim}
\affil{Department of Mathematical Sciences, Chalmers University of Technology and Gothenburg University, Gothenburg}
\date{\today}

\begin{document}

\maketitle

\begin{abstract}
In this paper we consider Monge-Amp\`ere equations on compact Hessian manifolds, or equivalently Monge-Amp\`ere equations on certain unbounded convex domains $\Omega\subseteq \mathbb{R}^n$, with a periodicity constraint given by the action of an affine group. In the case where the affine group action is volume-preserving, i.e., when the manifold is \emph{special}, the solvability of the corresponding Monge-Amp\`ere equation was established using the continuity method in \cite{chengyau1980}. In the general case we set up a variational framework involving certain dual manifolds and a generalization of the classical Legendre transform. We give existence and uniqueness results, elaborate on connections to optimal transport and quasi-periodic tilings of convex domains.
%In the case where the affine group action is volume-preserving, i.e., when the manifold is \emph{special}, the solvability of the corresponding Monge-Amp\`ere equation was established using the continuity method in \cite{chengyau1980}. 
%In the general case, we show that it is possible to define a generalization of the classical Legendre transform as an object living on a dual compact Hessian manifold. 

%The Legendre transform of this paper differs from the one \cite{shima1997geometry} in that it has a definition as a $\sup$. In particular, this gives a globally meaningful definition of the Legendre transform of non-convex objects, which is needed to make the variational formulation viable as a functional operating on continuous functions.
\end{abstract}

\tableofcontents

\section{Introduction}
Let $\Omega$ be an open convex subset of $\R^n$. Let $\Pi$ be a group that acts freely on $\Omega$ by affine transformations in such a way that there is a compact $K\subseteq \Omega$ satisfying $\Omega = \Pi K$, i.e. $M=\Omega/\Pi$ is a smooth compact manifold. Assume also that $\Omega$ admits a proper convex function, $\Phi_0$, such that its Hessian tensor 
$$ \frac{d^2\Phi_0}{dx_i dx_j} dx_i\otimes dx_j $$
is $\Pi$-invariant. The action on $\Omega\subset \R^n$ induces an action on $d\Phi_0(\Omega)\subset (\R^n)^*$, where $d\Phi_0:\Omega \rightarrow (\R^n)^*$ is the usual derivative of $\Phi_0$. This action is defined by the relation
\begin{equation}
d\Phi_0(x)=p \Leftrightarrow d\Phi_0(\gamma(x)) = \gamma.p 
\label{eq:action}
\end{equation}
where $p\in d\Phi(\Omega)$ and $\gamma\in \Pi$. Let $\mu$ and $\nu$ be locally finite $\Pi$-invariant measures\footnote{Throughout this paper all measures are assumed to be Borel} on $\Omega$ and $d\Phi(\Omega)$ respectively. Assuming $\mu$ has a density $f$ and $\nu$ has a density $g$ we will consider the equation

\begin{equation}
g\left(d\Phi(x)\right) \det\left(\Phi_{ij}(x)\right) = c f(x),
\label{eq:m-a_equation}
\end{equation}
for a suitable constant $c>0$. We will demand of a solution that it is convex and that its Hessian tensor is invariant under $\Pi$. We will say that an absolutely continuous measure $\mu = \rho dx$ is \emph{non-degenerate} if for any compact $K\subset \Omega$ it holds that $\rho \geq c_K > 0$. Recall also that a convex (not necessarily smooth) function, $\Phi$, is an \emph{Alexandrov solution} of \eqref{eq:m-a_equation} if the multivalued map $d\Phi$ satisfies
$$ \int_{A} \mu = c\int_{d\Phi(A)} \nu $$
for all measurable $A\subset \Omega$. Note that in this setting $\mu$ and $\nu$ does not have to be absolutely continuous.    
Our main theorem is
\begin{theorem}
\label{thm:main}
Let $\Omega$, $\Pi$ and $\Phi_0$ satisfy the conditions above. Assume that $\mu$ and $\nu$ are locally finite $\Pi$-invariant measures on $\Omega$ and $d\Phi(\Omega)$ respectively. Then there is a unique constant $c>0$ such that \eqref{eq:m-a_equation} admits an Alexandrov solution of the form $\Phi = \Phi_0+u$ for a $\Pi$-invariant function $u$. If $\nu$ is absolutely continuous with full support, then the solution is unique up to an additive constant. Moreover if $\mu$ and $\nu$ are both absolutely continuous with non-degenerate $C^{k,\alpha}$ densities, then $\Phi$ is $C^{k+2,\alpha}$.
\end{theorem}
%
%Theorem~\ref{thm:main} is a reformulation of a geometric result (Theorem~\ref{thm:monge-ampere}) regarding Hessian metrics on affine manifolds that minimize a certain Kantorovich type functional. The condition that $\Phi$ is strictly convex with $\Pi$-invariant Hessian in fact amounts to $\Phi$ defining a Hessian metric on the quotient manifold $M=\Omega/\Pi$. Moreover, by the main result in \cite{shima1981hessian} any compact affine manifold that admits a Hessian metric arises as a quotient of this type. In Section~\ref{sect:dual} we will give another definition of the dual action \eqref{eq:action}. From this definition it will be apparent that the action is affine. In fact we will show that the quotient define an Hessian manifold. We will denote this manifold $M^*$ and refer to it as the dual of $M$. The condition that $\mu$ and $\nu$ are $\Pi$-invariant ensures that $\mu$ defines a measure on $M$ and that $\nu$ defines a measure on $M^*$. A special case is given when the action of $\Pi$ preserves the Euclidean volume measure on $\R^n$. This will imply that the dual action preserves the Euclidean volume measure on $(\R^n)^*$. Then, fixing $\nu$ as the Euclidean volume measure on $(\R^n)^*$, equation \eqref{eq:m-a_equation} is equivalent to the inhomogenous Monge-Amp\`ere equation on special Hessian manifolds treated by Cheng and Yau in \cite{chengyau1980}.

\subsection{Geometric formulation}
Theorem~\ref{thm:main} is a reformulation of a geometric result (Theorem~\ref{thm:intro-geom}) regarding Monge-Amp\`ere equations on compact Hessian manifolds. To state it we will need some terminology from affine geometry. An affine manifold is a topological manifold $M$ equipped with a distinguished atlas $(U_i,x^i)$ such that the transition maps $x^i \circ (x^j)^{-1}$ are affine. Equivalently, an affine manifold is a smooth manifold equipped with a flat torsion-free connection $\nabla$ on $TM$. The coordinates in the distinguished atlas are often referred to as \emph{affine coordinates} on $M$. A function, $f$, on $M$ is said to be affine (or convex) if it is affine (convex) in the affine coordinates or, equivalently, its second derivative with respect to $\nabla$
\begin{equation} 
\nabla df = \frac{d^2f}{dx_i dx_j}dx_i\otimes dx_j \label{eq:hessian} 
\end{equation}
vanishes (is semi-positive). A Hessian metric on an affine manifold $M$ is a Riemannian metric $g$ which is locally of the form \eqref{eq:hessian}. In other words, there is a covering $\{U_i\}$ of $M$ and smooth functions $\{\phi_i: U_i\rightarrow \R\}$ such that 
\begin{equation}
	g = \nabla d\phi_i.
	\label{eq:hessian-metric}
\end{equation}
A Hessian manifold, $(M,\{\phi_i\})$, is an affine manifold $M$ together with a Hessian metric $\{\phi_i\}$. For short we write $\phi$ instead of the collection $\{\phi_i\}$. Note that as a consequence of \eqref{eq:hessian-metric}, we have that $\phi_i - \phi_j$ is locally affine where it is defined. We will explain in Section~\ref{sec:geometric} how the data $\{\phi_i - \phi_j\}$ define a principal $\R$-bundle $L\to M$ that respects the affine structure on $M$ (affine $\R$-bundle for short). We will say that Hessian metrics defining the same affine $\R$-bundle lie in the same \emph{K\"ahler class}, and will occasionally refer to a Hessian manifold only using the data $(M,L)$ without giving reference to a specific Hessian metric.

An affine manifold is special if the transition maps preserve the Euclidean volume form on $\R^n$ or, equivalently, if the holonomy associated to to $\nabla$ sits in $\SL(n,\R)$. An important property of special Hessian manifolds is that the real Monge-Amp\`ere measure of the Hessian metric
\begin{equation}
\label{eq:special_ma}
\MA(\phi) = \det \left(\frac{d^2\phi}{dx_i dx_j}\right) dx_1\wedge \ldots \wedge dx_j 
\end{equation}
is invariant under coordinate transformations and globally defines a measure on $M$. Indeed, differential equations involving this operator has been studied in a number of papers. Existence and uniqueness for associated Monge-Amp\`ere equations on special Hessian manifolds were first given by Cheng and Yau \cite{chengyau1980}. Another notable work is \cite{caffarelli2001regularity}. In this paper we will explain that, although the expression in \eqref{eq:special_ma} is only well-defined when $M$ is special, it is possible to, by fixing an absolutely continuous measure $\nu$ on a certain dual manifold, define a Monge-Amp\`ere operator on general Hessian manifolds. More precisely, we will explain that the data $(M,L)$ defines a dual Hessian manifold $M^*$. This is the same construction found in the literature on the Strominger-Yau-Zaslow picture of mirror symmetry (see for example pages 428--429 in \cite{mirrorsymmetry}). Given a measure $\nu$ on $M^*$ and a Hessian metric $\phi$ on $M$ we define a $\nu$-Monge-Amp\`ere measure $\MA_\nu(\phi)$ (see Definition~\ref{def:numa}) and consider the equation 
\begin{equation}
\label{eq:numaeq}
\MA_\nu(\phi) = \mu. 
\end{equation}
for measures $\mu$ on $M$. We will also introduce a concept of weak solutions to this equation. The majority of Section~\ref{sec:ma} is devoted to the proof of the following main theorem. 
\begin{restatable}{theorem}{geom-mav}
	Let $(M,\phi_0)$ be a compact Hessian manifold. Let $\mu$ and $\nu$ be probability measures on $M$ and $M^*$, respectively. Then there is a continuous function $u$ on $M$, such that $\phi = \phi_0 + u $ solves \eqref{eq:numaeq} in the weak sense. If $\nu$ is absolutely continuous and $\phi_0$ and $\phi_1$ are solutions to \eqref{eq:numaeq}, then $\phi_1-\phi_0$ is constant. If, in addition, $\mu$ and $\nu$ are absolutely continuous with non-degenerate $C^{k,\alpha}$-densities for some $k\in \mathbb{N}$ and $\alpha\in (0,1)$, then the solution is $C^{k+2,\alpha}$.
	\label{thm:intro-geom}
\end{restatable}
\begin{remark}
The constant $c$ in Theorem~\ref{thm:main} is determined by the fact that $\mu$ and $c\nu$ should define measures of equal mass on $\Omega/\Pi$ and $\Omega^*/\Pi$. In Theorem~\ref{thm:intro-geom} this obstruction is handled by demanding that both $\mu$ and $\nu$ are probability measures.
\end{remark}
\begin{remark}
    In \cite{chengyau1980} Cheng and Yau consider a certain type of Monge-Amp\`ere equations on non-special Hessian manifolds, namely equations of the form
    $$ \det(\phi_{ij}) = \rho^2 $$
    where $\rho$ is a density on $M$. In other words, they consider equations involving the expression $\det(\phi_{ij})$ which transforms as the square of a density on $M$. We stress that our approach is different. The $\nu$-Monge-Amp\`ere defines a measure on $M$ regardless if $M$ is special or not.
\end{remark}

It will follow from the construction that if $M$ is special then $M^*$ is special. If we choose $\nu$ as the canonical $\nabla$-parallel measure on $M^*$ then \eqref{eq:numaeq} reduces to the standard inhomogenous Monge-Amp\`ere equation on special manifolds considered in the literature. 
%The dual manifold $M^*$ can be found in the literature. It is diffeomorphic to $M$ and equivalent, as Hessian manifolds, to $M$ equipped with the dual Hessian structure considered by Shima (see \cite{shima2007geometry}). 

Finally, we remark that the local geometry of smooth measured metric spaces of the form $(M,\nabla d\phi,\mu)$ where $\phi$ and $\mu$ are related as in Theorem~\ref{thm:intro-geom}, have recently been studied by Klartag and Kolesnikov in \cite{klartag2016remarks}. It is interesting to note that our approach shows that a pair of measures $(\mu,\nu)$ with smooth densities on $M$ and $M^*$ determines a pair of measured metric spaces $(M,\nabla d\phi,\mu)$ and $(M^*,\nabla^* d\phi^*,\nu)$ of the form studied in \cite{klartag2016remarks} related by Legendre transform.

Let us finally point out that one of the motivations for developing the present approach to global Monge-Amp\`ere equations comes from Mirror Symmetry and tropical geometry (in particular the framework of the Strominger-Yau-Zaslow, Gross-Wilson and Kontsevich-Soibelman conjectures \cite{mirrorsymmetry}). In this framework dual affine (singular) manifolds appear as the ''large complex limits'' of ''mirror dual'' complex/symplectic manifolds and the corresponding K\"ahler-Einstein metrics (solving complex Monge-Amp\`ere equations) are expected to converge to solutions of real Monge-Amp\`ere equations on the singular affine manifolds in question. Hopefully, the present approach can be extended to such singular (and possibly non-compact) affine manifolds, but we leave this challenging problem for the future.

\subsection{Optimal Transport Interpretation}
One of the key points of the present paper is to show that equation \eqref{eq:numaeq} fits nicely into the theory of optimal transport. Recall that an optimal transport problem is given by two probability spaces $(X,\mu)$ and $(Y,\nu)$ together with a cost function $c:X\times Y\rightarrow \R$. We will explain in Section~\ref{sec:pairing} how the data $(M,L)$ determines a cost function $c=c_{(M,L)}:M\times M^*\rightarrow \R$. This means a Hessian manifold $(M,L)$ together with two measures $\mu$ and $\nu$ on $M$ and $M^*$ respectively determines an optimal transport problem. Moreover, by construction, the differentials of $\{\phi_i\}$, $x\mapsto d\phi_i|_x$, induces a diffeomorphism, which we will denote $d\phi$, from $M$ to $M^*$. We have the following theorem with respect to this interpretation.
\begin{restatable}{theorem}{opttransport}
Let $(M,L)$ be a compact Hessian manifold. Let $\mu$ and $\nu$ be probability measures on $M$ and $M^*$ respectively. Assume $\phi$ is a smooth strictly convex section of $L$ such that 
$$ \MA_\nu(\phi) = \mu. $$
Then $d\phi$ is the optimal transport map determined by $M,M^*,\mu,\nu$ and the cost function induced by $(M,L)$.
\label{thm:optimaltransport}
\end{restatable}

In the classical case of optimal transport, when $X=\R^n$ and $Y=(\R^n)^*$, the cost function is given by $-\langle \cdot,\cdot\rangle$ where $\langle \cdot,\cdot \rangle$ is the standard pairing of $\R^n$ with $(\R^n)^*$. Our setting is a generalization of this in the sense that the cost function induced by a Hessian manifold $(M,L)$ is induced by a pairing-like object $[\cdot,\cdot]$. However, $[\cdot,\cdot]$ will not be a bi-linear function on $M\times M^*$. Instead it will be a (piecewise) bi-linear section of a certain affine $\R$-bundle over $M\times M^*$. 

Moreover, we will explain in Section~\ref{sec:pairing} that if $(M,L)$ is special then $L$ determines a flat Riemannian metric on $M$. Moreover, it turns out that when $(M,L)$ is special, $M$ and $M^*$ are equivalent as affine manifolds. We will show that under this identification the induced cost function (defined on $M\times M^*$) is given by the squared distance function determined by a certain flat Riemannian metric on $M$, hence proving
\begin{restatable}{theorem}{secondopttransport}
\label{thm:optimaltransportsecond}
Let $(M,L)$ be a compact special Hessian manifold, $\mu$ and $\nu$ probability measures on $M$ and $M^*$ respectively. Then equation~\eqref{eq:numaeq} is equivalent to the optimal transport problem determined by $\mu$, $\nu$ and $d^2/2$, where $d$ is the flat Riemannian metric on $M$ induced by $L$.   
\end{restatable}
In this sense our setting can also be seen as a generalization of the works by Cordero-Erasquin \cite{cordero1999} and McCann \cite{mccann2001} on optimal transport on Riemannian manifolds.

\subsection{The Legendre Transform}

To formulate the Kantorovich type functional, we generalize the Legendre transform from $\mathbb{R}^n$ to the setting of Hessian metrics on affine manifolds. A Legendre transform of Hessian metrics on manifolds has appeared elsewhere in the literature, see, e.g., \cite{shima1997geometry,mirrorsymmetry}. In this setting, the Legendre transform is formulated in terms of the flat torsion-free connection $\nabla$ of the tangent bundle $TM$. It is shown that the connection $\nabla^* = 2\nabla_\phi - \nabla$, where $\nabla_\phi$ denotes the Levi-Civita connection defined by the Hessian metric is also a flat torsion-free connection on $TM$, defining a dual affine structure on $M$. 

We attempt to take a more global approach to constructing the Legendre transform on a Hessian manifold $(M,\phi)$. The crucial observation is that the affine structure on $M$ allows one to define local affine functions (or more generally, affine sections to the principal $\R$-bundle $L\to M$ defined by $\phi$) on $M$, which in turn can be used to define the Legendre transform by a supremum formula. A difficulty lies in that generally an affine manifold does not allow any global non-trivial affine sections. In this paper this is dealt with by passing to universal cover of the compact Hessian manifold $(M,\phi)$, which by \cite{shima1981hessian} can be realized as a convex set $\Omega\subset \R^n$ with a convex exhaustion function $\Phi$. The key advantage of this approach, compared to that of \cite{shima1997geometry,mirrorsymmetry}, is that the supremum formula allows the definition of a projection operator $P$ mapping continuous sections to convex sections. To illustrate this point, we note that the Legendre transform in \cite{shima1997geometry,mirrorsymmetry}, being defined as a change of connection on $TM$, is purely local, and in $\R$ reduces to the expression (for a smooth strictly convex $\phi:\R\to \R$)

\begin{equation}
	\phi^{*}(\phi'(x)) = \phi'(x)x - \phi(x).
	\label{eq:legendre-r}
\end{equation}
However, issues arise when attempting to take the Legendre transform of a non-convex function $f$, the one immediately relevant for our purposes being that $f^{**}$ does not define a projection operator from the space of continuous functions on $\R$ to convex functions. However, the slight modification (sometimes called the Legendre-Fenchel transform) of the above expression to

\begin{equation}
	\phi^*(p) = \sup_x px - \phi(x)
	\label{eq:legendre-fenchel-r}
\end{equation}
allows immediately the definition of the projection $f \mapsto f^{**}$. A main contribution of this paper is that we generalize \eqref{eq:legendre-fenchel-r} instead of \eqref{eq:legendre-r}, giving us such a projection operator. It is this projection operator that allows us to give a variational formulation of the Monge-Amp\'ere equation, formulated in terms of a Kantorovich functional with continuous functions as domain. 

Using the variational formulation, the existence and uniqueness of solutions is reduced to a question regarding existence and uniqueness of minimizers of functionals, and a main result in this (which implicitly can also be found in \cite{chengyau1980}) is a compactness result for Hessian metrics in a fix K\"ahler class. 

\begin{restatable}{theorem}{thmCompact}
    Let $(M,L)$ be a compact Hessian manifold. Then the space of convex sections of $L$ modulo $\R$ is compact, in the topology of uniform convergence modulo $\R$.
    \label{thm:compactness}
\end{restatable}

\subsection{Further results}
Using Theorem~\ref{thm:compactness}, we outline in Section~\ref{sec:KE} how functionals mimicking the Ding- and Mabuchi functionals in complex geometry can be shown to have minimizers, this also giving existence and uniqueness results for a K\"ahler-Einstein-like equation on Hessian manifolds. The main theorem in this regard can be formulated as follows.

\begin{restatable}{theorem}{thmKE}
    Let $(M,L,\phi_0)$ be a compact Hessian manifold, let $\nu$ be an absolutely continuous probability measure of full support on $M^*$, let $\mu$ be a probability measure on $\mu$ and let $\lambda\in \R$. Then the equation
    
    \begin{equation}
        \MA_\nu \phi = e^{-\lambda (\phi - \phi_0)} \mu
        \label{eq:KE-hess}
    \end{equation}
    has a solution. 
\label{thm:KE-hess}
\end{restatable}

We wish to point out that in contrast to the complex setting solutions to \eqref{eq:KE-hess} do not define Einstein metrics, in the sense that \eqref{eq:KE-hess} is not a reformulation of the Einstein equation $\Ric g = \lambda g$. However, as mentioned above the geometric properties of solutions to equation \eqref{eq:KE-hess} have very recently been studied by Klartag and Kolesnikov in \cite{klartag2016remarks}. Moreover, when $M=\R^n$, \eqref{eq:KE-hess} has been studied as a twisted K\"ahler-Einstein equation on a corresponding toric manifold (see \cite{wangzhu,berman2013real}) and when $M$ is the real torus with the standard affine structure \eqref{eq:KE-hess} has been studied as an analog of a twisted singular K\"ahler-Einstein equation in \cite{hultgren}.
In the case when $\lambda>0$ we will also show uniqueness of solutions to \eqref{eq:KE-hess}. When $\lambda<0$ solutions are not unique in general. Nevertheless, with the variational approach outlined here one gets a set of distinguished solutions, namely the minimizing ones.

Further, although this paper is chiefly concerned with the case where $M = \Omega/\Pi$ is a manifold, we in Section~\ref{sec:orbifold} outline how the results can be extended to an orbifold setting. 

\subsection{Atomic measures}

We also include a section on atomic measures, and show that the only convex sections $\phi$ where the Monge-Amp\'ere operator has finite support are the piecewiese affine ones (Theorem \ref{thm:pw-affine}).

Corresponding to a piecewise affine section of $L$ is a (locally) piecewise affine function $\Phi$ on $\Omega$. The singular locus of $\Phi$ defines a quasiperiodic tiling of $\Omega$ (with respect to $\Pi$) by convex polytopes. This means that solving Monge-Amp\'ere equations with atomic data corresponds to finding quasi-periodic tilings of the covering space. In the case of real tori $M = \mathbb{R}^n/\mathbb{Z}^n$, for $n=2,3$ this is related to the computational work in \cite{caroli2009computing,caroli2008computation,yan2011computing}.

The main points of this section are the following theorems.

\begin{restatable}{theorem}{pwaffthm}
We call a probability measure $\mu$ on $M$ atomic if $\mu = \sum_{i=1}^N \lambda_i \delta_{x_i}$. Let $\nu$ be an aboslutely continuous probabilty measure of full support on $M^*$. Then

\begin{equation}
	\MA_\nu \phi \text{ is atomic}  \Leftrightarrow \phi \, \text{is piecewise affine.}
\end{equation}

\label{thm:pw-affine}
\end{restatable}

\begin{restatable}{theorem}{pwaffcor}
	Any Hessian metric $\phi_0$ on a compact Hessian manifold $(M,L,\phi_0)$ can approximated uniformly by a piecewise affine section.
	\label{thm:pw-approx}
\end{restatable}

\begin{remark}
    We point out that the above two theorems seem to be a phenomenon specific to the compact Hessian setting, in the sense that the corresponding statements are false both in $\R^n$ and on compact K\"ahler manifolds. In $\R^n$, $n\geq 2$ we may take $\phi = \|x\|$, which is not piecewise affine, but where $\MA \phi = \delta_0$. 
    
    Further, if the Monge-Amp\'ere measure of a $\omega$-plurisubharmonic function $u$ on a compact K\"ahler manifold $(X,\omega)$ is discrete, (see \cite{coman2009quasiplurisubharmonic}), the current $\omega + dd^c u$ does not necessarily vanish. To see this, one can take $X$ to be complex projective space $\mathbb{P}^n$, and letting $\omega$ correspond to $dd^c \log |z|^2$ on a dense embedding $\mathbb{C}^n \subset \mathbb{P}^n$. Then $\omega$ is $\mathbb{C}^*$ invariant, and descends to the Fubini-Study form on $\mathbb{P}^{n-1}$. Hence $\omega \neq 0$ on the dense set $\mathbb{C}^n\subset \mathbb{P}^n$ away from the origin, but the Monge-Amp\'ere mass is concentrated on $0$. 
\end{remark}

Also note that Theorem~\ref{thm:pw-approx} can be seen as analogous to an approximation result in \cite{demailly1992regularization}, stating that an $\omega$-plurisubharmonic function on a compact K\"ahler manifold $(X,\omega)$ can be written as a decreasing sequence of $\omega$-plurisubharmonic functions with analytic singularities. However we obtain uniform convergence instead. To the best of our knowledge this is the first such result in the setting of Hessian manifolds.

%In the case of atomic measures, where we are able to show that the corresponding Monge-Ampere solutions are piecewise affine, the solution to \eqref{eq:m_a-equation} corresponds to a quasi-periodic tiling by polyhedrons of convex sets. As a special case, one can take the real torus $(S^1)^2 = \mathbb{R}^2/\mathbb{Z}^2$, which carries a natural Hessian metric induced by the Euclidean metric on the cover $\mathbb{R}^2$, whereby solving a Monge-Ampere equation with atomic measures yields a bona fide $\mathbb{Z}^2$-periodic tiling of the plane by convex polyhedra. Extensive numerical work regarding Delaunay triangulations of various spaces... (mer referenser här till Bogdanov, Teillaud, Caroli)

%\section{Preliminaries}
%\input{prelim.tex}

\section{Geometric Setting}
\label{sec:geometric}
\begin{definition}[affine $\R$-bundle]
A affine $\R$-bundle over an affine manifold $M$ is an affine manifold $L$ and a map $\tau: L\rightarrow M$ such that the fibers of $\tau$ have the structure of affine manifolds isomorphic to $\R$ and such that there is a collection of local trivializations $\{(U_i,p_i)\}$ such that the transition maps $p_i\circ p_j^{-1}: U_j\cap U_i\times \R \rightarrow U_i\times \R$ are of the form 
\begin{equation}
	(x,y)\mapsto (x',y+\alpha_{ij}(x))
	\nonumber
\end{equation}
for some affine transition functions $\alpha_{ij}$ on $U_i\cap U_j$. 
\end{definition}
\begin{remark}
	It follows that an affine $\R$-bundle is a principal $\R$-bundle compatible with the affine structure on $M$.
\end{remark}
A section $s:M\rightarrow L$ of an affine $\R$-bundle is affine (or convex) if it is represented by affine (convex) functions in the trivializations. 

Note that if $g$ is a Hessian metric on $M$ induced by $\{\phi_i\}$, then \eqref{eq:hessian-metric} implies that $\phi_i - \phi_j$ is affine for any $i,j$. Putting $\alpha_{ij}=\phi_i-\phi_j$ defines an affine $\R$-bundle over $M$ in which $\{\phi_i\}$ is a convex section. We will often refer to a Hessian manifold as $(M,L,\phi)$ where $L$ is the affine $\R$-bundle associated to $\{\phi_i\}$ and $\phi$ is the convex section in $L$ defined by $\{\phi_i\}$. We will also refer to $\phi$ both as a weak Hessian metric, and as a convex section to $L$ interchangeably. We will say that $L$ is positive if it admits a smooth and strictly convex section. This is consistent with the terminology used in the complex geometric setting, as well as the tropical setting \cite{mikhalkin} Also, in analogy with the setting of K\"ahler manifolds we make the following notational definition.

\begin{definition}
    If $\phi$ and $\phi_0$ are convex sections to the same affine $\R$-bundle $L\to M$, we say that $\phi$ lies in the K\"ahler class of $\phi_0$. 
\end{definition}

Let $\pi:\Omega\rightarrow M$ be the universal covering of $M$. By pulling back $\nabla$ with the covering map we get that $\Omega$ is also an affine manifold. The pullback of $L$ defines an an affine $\R$-bundle over $\Omega$. Let us denote this bundle $K$ and let $\pi^*\phi$ be the pullback of $\phi$ to $K$. Let $\Gamma(\Omega,K)$ be the space of global affine sections in $K$. We have the following basic
\begin{proposition}
\label{pr:extension_of_sections}
	Any local affine section of an affine $\R$-bundle over a simply connected manifold $\Omega$ may be uniquely extended to a global affine section.  
\end{proposition}

\begin{proof}
	Assume $s$ is defined in a neighborhood of $x\in \Omega$. To define $s(y)$ for $y\in \Omega$, let $\gamma$ be a curve in $\Omega$ from $x$ to $y$. Cover $\gamma$ with open balls $B_i$ each contained in a some local trivialization of $L$. In each ball there is a unique way of extending $s$. Moreover, replacing $\gamma$ with a perturbation of $\gamma$ allow us to use the same cover, $\{B_i\}$. This means, since $\Omega$ is simply connected, that $s(y)$ does not depend on $\gamma$. 
\end{proof}
 
Proposition~\ref{pr:extension_of_sections} says that $\Gamma(\Omega,K)$ is isomorphic (as an affine manifold) to the space of affine functions on $\R^n$, $(\R^n)^*\times \R$ (see Remark~\ref{rem:identification}). In particular $\Gamma(\Omega,K)$ is nonempty. 

\begin{remark}
If $y_1$ and $y_2$ are two points in the same fiber of an affine $\R$-bundle, then, since the structure group acts additively, their difference, $y_1-y_2$, is a well defined real number. Consequently, if $s_1$ and $s_2$ are sections of an affine $\R$-bundle over a manifold $M$, then $s_1-s_2$ defines a function on $M$.  Generalizing this observation to sections $s_1,s_2$ of the affine $\R$-bundles $L_1,L_2$, we see that the set of affine $\R$-bundles over $M$ naturally carries the structure of an $\R$ vectorspace. 
\end{remark}

Taking $q\in \Gamma(\Omega,K)$ we may consider the pullback $\pi^*\phi$ to $K$ and  
\begin{equation}
	\Phi_{q} = \pi^*\phi - q 
	\nonumber
\end{equation}
This is, since both $\pi^*\phi$ and $q$ are sections of $K$, a well-defined function on $\Omega$. Moreover, $\nabla d\Phi_q = \nabla d \tilde\phi$. This means the Hessian of $\Phi_q$  is strictly positive and defines the same metric as the one given by the pull back of the Hessian metric $\nabla d\phi$ on $M$. We conclude that any Hessian metric on an affine manifold may be expressed as the Hessian of a \emph{global} function on the covering space. Now, by a theorem by Shima \cite{shima1981hessian}, the covering space of any compact Hessian manifold may be embedded as a convex subset in $\R^n$. Convexity of the covering space implies that $\Phi_q$ is convex. Moreover, it readily follows from the proof in \cite{shima1981hessian} that, for some choice of $q_0$, $\Phi_{q_0}$ is an exhaustion function of $\Omega$.

\subsection{A dual Hessian manifold}
\label{sec:dual}
In the notation of the previous section we have 
\begin{center}
\begin{tikzcd}
	K \arrow{d} \arrow{r} & \Omega \arrow{d}{\pi} \\
	L \arrow{r}{\tau} & M
\end{tikzcd}
\end{center}
where $\Omega$ is the universal covering space of $M$ and $K$ is the pullback of $L$ under the covering map. In this section we will define a dual diagram
\begin{center}
\begin{tikzcd}
	K^* \arrow{d} \arrow{r} & \Omega^* \arrow{d} \\
	L^* \arrow{r} & M^*
\end{tikzcd}
\end{center}
with dual objects $K^*$, $\Omega^*$, $L^*$ and $M^*$ where $M^*$ will turn out to give (under suitable assumptions) another Hessian manifold which we will refer to as the \emph{dual Hessian manifold}. 
\begin{definition}
Let $K^*$ be the subset of $\Gamma(\Omega,K)$ given by all $q\in \Gamma(\Omega,K)$ such that $\Phi_q:\Omega\rightarrow  \R$ is bounded from below and proper.
\end{definition}
\begin{remark}
If $M=\R^n$ and $L$ is the trivial affine $\R$-bundle $\R^n\times \R$, then $\phi$ is a strictly convex function on $\R^n$ and $K^*$ is given by the affine functions on $\R^n$ such that their derivative is in the gradient image of $\phi$. 
\end{remark}
\begin{lemma}
The set $K^*\subset \Gamma(\Omega,K)$ is nonempty and open. Moreover, it does only depend on $(M,L)$. 
\end{lemma}
\begin{proof}
As mentioned in the end of the previous section, by \cite{shima2007geometry}, $\Phi_q$ is an exhaustion function for a suitable choice of $q$. This means $K^*$ is nonempty. To see that $K^*$ is open, assume $q\in K^*$ and note that since $\Phi_q$ is bounded from below and proper it admits a minimizer $x_0\in \Omega$. Let $U$ be a neighbourhood if $x_0$. Since $\Phi_q$ is strictly convex $\Phi_q(x)>\epsilon|x-x_0|-C$ outside $U$. We have that for any $q'$ close to $q$, $|q-q'|<\epsilon|x-x_0|/2+C'$ for some $C'$. This means  
$$ \Phi_{q'} = \Phi_q + q-q' > \Phi_q - \frac{\epsilon}{2}|x-x_0| - C' > \frac{1}{2}\Phi_q -C/2 -C' $$
outside $U$. Since $\Phi_q/2$ is proper and bounded from below if and only if $\Phi_q$ is proper and bounded from below it follows that $\Phi_{q'}$ is proper and bounded from below, hence $q'\in K^*$. 

Finally, let $\phi$ and $\psi$ be two Hessian metrics of the same affine $\R$-bundle. Then $\Phi_q-\Psi_q = \pi^*\phi-\pi^*\psi$ is a continuous function on $\Omega$ that descends to $M$. This means it is bounded. We conclude that $\Phi_q$ is bounded from below and proper if and only if $\Psi_q$ is bounded from below and proper.   
\end{proof}

Note that, given $C\in \R$, we may consider the map on $\Gamma(\Omega,K)$ given by 
\begin{equation}
q\mapsto q+C. 
\label{eq:actionKstar}
\end{equation}
This defines a smooth, free and proper action by $\R$ on $\Gamma(\Omega,K)$.
Moreover, $\Phi_q$ is proper if and only if $\Phi_{q+C} = \Phi_q - C$ is proper, hence the action preserves $K^*$. 
\begin{definition}
We define $\Omega^*$ to be the quotient $K^*/\R$.
\end{definition}

\begin{remark}
\label{rem:identification}
    We here give a way to explicitly identify $\Omega$ and $\Omega^*$ with compatible embeddings in $\R^n$ and $(\R^n)^*$, respectively. Fixing a point $q_0\in K^*$, we may write any $q\in K^*$ as $q = q_0 + (q-q_0)$. Since $q - q_0$ is an affine function this yields an identification $\Gamma(\Omega,K) \overset{q_0}{\simeq} \Gamma(\Omega,0)$, where $0$ denotes the trivial affine $\R$-bundle over $\Omega$. Further, choosing a point $x_0 \in \Omega$ and a basis for $T_{x_0} \Omega$ yields an identification of $\Omega$ with an embedding to $i_1: \Omega \to \R^n$, and thus also an identification $\Gamma(\Omega,0) \overset{T_{x_0}\Omega}{\simeq} \Gamma(\R^n,0) \simeq (\R^n)^* \times \mathbb{R}$. This provides an embedding $i_2: \Omega^* \to (\R^n)^*$. In fact, as will be explained later, $d(\pi^* \phi)$ yields a map $\Omega\to \Omega^*$, and the identification can be summarized as saying that the following diagram commutes.
    
\begin{center}
\begin{tikzcd}
	\Omega \arrow{d}{i_1} \arrow{r}{d(\pi^* \phi)}  & \Omega^* \arrow{d}{i_2} \\
	i_1(\Omega) \arrow{r}{d\Phi_{q_0}} & i_2(\Omega^*)
\end{tikzcd}
\end{center}

\end{remark}

\begin{lemma}
The quotient map $ K^*\rightarrow \Omega^* $ gives $K^*$ the structure of an affine $\R$-bundle over $\Omega^*$. 
\end{lemma}
\begin{proof}
First of all, note that the fibers of the quotient map are affine submanifolds of $K^*$ isomorphic to $\R$. Moreover, there is a global affine trivialization of $K^*$ over $\Omega^*$. To see this, recall that by Remark~\ref{rem:identification} $K^*$ is isomorphic to a subset of $(\R^n)^*\times \R$. The action on $K^*$ given by \eqref{eq:actionKstar} extends to all of $(\R^n)^*\times \R$ where it is given by $(a,b)\rightarrow (a,b+C)$. In particular, the quotient map is the same as the projection map on the first factor. We conclude that the identification of $K^*$ with the subset of $(\R^n)^*\times \R$ defines a global trivialization of $K^*$ over $\Omega^*$. 
\end{proof}

%Here we show how to define a dual object to a Hessian manifold, which inherits many of the properties of the original manifold. In particular it a Hessian manifold, which will have to wait until \ref{prop:dual-is-manifold}, however we will for ease of notation refer to the dual objects as manifolds already here. Motivated by the Legendre transform, seen in $\mathbb{R}^n$ as a map from differentials of affine functions to $\mathbb{R}$, we will here construct the dual manifold as a quotient of (a subset of) $\Gamma(\tilde{M},q^*L)$ to a logline bundle over it. More precisely, the dual will be the 'gradient image' of $\phi$, though of as a map from $M$ to equivalence classes of $\Gamma(M,q^*L)/\mathbb{R}$. To make the construction rigorous, we will take detour by viewing the sections $\phi$ as convex \emph{functions} on $\tilde{M}$.

%Let $q:\tilde M\rightarrow M$ be the universal covering of $M$. Note that, by pulling back the flat connection on $M$ to $tilde M$, we get an affine structure on $\tilde M$. Let $L$ be an affine logline bundle over $M$. Then the pull back $q^* L$ is an affine logline bundle over $\tilde M$. 
%\begin{definition}[The dual universal cover]
	%Let $(M,L,\phi)$ be a Hessian manifold. We let 
%\begin{align}
	%\Omega^* & = \nabla(\pi^*\phi)(\tilde M) \\
	%K^* &:= \tau^{-1} (\nabla(\pi^*\phi)(\tilde M))
%\end{align}
%
%\end{definition}
%

Now, let $\Pi$ be the fundamental group of $M$, acting on $\Omega$ by deck transformations. This action extends to an action on $K$. To see this, note that the total space of $K$ can be embedded in $\Omega\times L$ as the submanifold 
$$ \{ (x,y)\in \Omega\times L:\pi x= \tau y \}. $$
The action by $\Pi$ on $K$ is then given by $\gamma(x,y) = (\gamma x,y)$. If $q$ is an affine section of $K$ then its conjugate $\gamma \circ q \circ \gamma^{-1}$ is also an affine section of $K$. We get an action of $\Pi$ on $\Gamma(\Omega,K)$ defined by
$$ \gamma.q = \gamma\circ q \circ \gamma^{-1}. $$

\begin{lemma}\label{lemma:action1}
The action by $\Pi$ on $\Gamma(\Omega,K)$ commutes with the action by $\R$. Moreover, if $\phi$ is a convex section of $L$ and $q\in \Gamma(\Omega,K)$ then the action satisfies
$$ \Phi_{\gamma.q} = \Phi_q\circ \gamma^{-1}. $$
Finally, $q\in K^*$ if and only if $\gamma.q\in K^*$.
\end{lemma}
\begin{proof}
First of all, if we have two points in the same fiber of $K$, $(x,y_1)$ and $(x,y_2)$. Then 
\begin{equation} 
\gamma (x,y_1)-\gamma (x,y_2) = (\gamma x, y_1) - (\gamma x, y_2) = y_1-y_2 = (x,y_1)-(x,y_2).
\label{eq:differencespreserved}
\end{equation}
In particular, if $q_1,q_2\in \Gamma(\Omega,K)$ then $q_1=q_2+C$ if and only if $\gamma q_1 = \gamma q_2+C$. This proves the first point of the lemma. Moreover, by \eqref{eq:differencespreserved} we have 
$$ \gamma \circ (\pi^*\phi) - \gamma\circ q = (\pi^*\phi) -q. $$
Since $\pi^*\phi$ descends to a convex section of $L$ we have $\gamma\circ (\pi^*\phi)\circ \gamma^{-1} = \pi^*\phi$. This means 
$$ \Phi_{\gamma.q} = \pi^*\phi - \gamma\circ q\circ\gamma^{-1} = (\pi^*\phi) \circ \gamma^{-1} -q\circ\gamma^{-1} = \Phi_q\circ\gamma^{-1} $$
proving the second point of the lemma. For the last point of the lemma, note that $\Phi_q$ is bounded from below if and only if $\Phi_q\circ \gamma^{-1}$ is bounded from below. Moreover, any invertible affine transformation of $\R^n$ is proper and has proper inverse. This means $\Phi_q$ is proper if and only if $\Phi_q\circ \gamma^{-1}$ is proper. 
\end{proof}
Form the first and third point of Lemma~\ref{lemma:action1} we have that $\Pi$ acts on $K^*$ and $\Omega^*$.  
\begin{definition}\label{def:dual}
	We define 
\begin{align}
	L^* &= K^*/\Pi \nonumber \\
	M^* &= \Omega^*/\Pi. \nonumber
\end{align}
\end{definition}
\begin{remark}
It is clear from the definition that the actions by $\Pi$ on $K^*$ and $\Omega^*$ are affine. However, at this point it is not clear that they are free. We will prove in the next section that $K$ and $K^*$ are diffeomorphic and that the action on $K$ and the action on $K^*$ are the same up to conjugation. This will imply that the quotients in Definition~\ref{def:dual} are affine manifolds. 
\end{remark}

%\begin{proposition}
%The gradient map $\nabla(\pi^*\phi):\tilde M\rightarrow \Omega^*$ descends to a map from $M$ to $M^*$.
%\end{proposition}

%\begin{proof}
%Assume $\nabla(\pi^*\phi)(x) = p\in K^*$. This means there is an element $\boldmath p$ in the fiber above $p$ which is (as an affine section) tangential to $\phi$ at $x$. If $\gamma$ is an element in $\pi_1(M)$, then $\gamma.\boldmath p$ is tangential to $\gamma.\phi$ at $\gamma. x$, hence $\nabla (\pi^*\phi)(\gamma.x) = \gamma.p$.
%\end{proof}

In a lot of examples $\Omega$ and $\Phi_q$ are explicit. The action of $\Pi$ on $K^*$ is then explicitly described by
\begin{lemma}\label{lemma:action2}
Let $\gamma\in \Pi$ and $q \in K^*$. Then  
$$ \gamma.q = q + \Phi_q-\Phi_q\circ \gamma^{-1}. $$
\end{lemma}

\begin{proof}
From the second point of Lemma~\ref{lemma:action1} we get
$$ \Phi_q-\Phi_q\circ \gamma^{-1} = \Phi_q-\Phi_{\gamma.q} = \gamma.q-q $$
proving the lemma.
\end{proof}

\begin{example}
Let $M=\R^n$, $L$ be the trivial affine $\R$-bundle, $\R^n\times \R$ over $M$ and $\phi$ be any smooth strictly convex function on $M$. Then $\Pi$ is trivial, $M^* = \Omega^* = d\phi(M)$ and $L^*$ is the trivial affine $\R$-bundle, $M^*\times \R$, over $M^*$.
\end{example}
\begin{example}
Let $M$ be the standard torus $\R^n/\mathbb{Z}^n$. Let $\phi$ and $L$ be the data defined by the Euclidean metric on $M$, in other words $\Phi_{q_0} = |x|^2/2$ for some $q_0\in \Gamma(\Omega,K)$. Now, any $q\in \Gamma(\Omega,K)$ is given by $q_0+A$ for some affine function $A=\langle x ,a \rangle + b$ on $\Omega$ ($a\in \R^n$, $b\in \R$). This means $\Phi_q = \Phi_{q_0} - \langle x ,a \rangle - b$ is bounded from below and proper for all $q\in \Gamma(\Omega,K)$ and we get that $K^*=\Gamma(\Omega,K)\cong \R^n\times \R$. The deck transformations acts by lattice translations. Given a deck transformation $\gamma_m: x\mapsto x + m$ for $m\in \mathbb{Z}^n$ we calculate $\gamma_m.q$ by 
\begin{eqnarray}
	\gamma_m.q & = & q + \Phi_q - \Phi_q \circ \gamma^{-1}
	\nonumber \\
	 & = & q + \frac{|x|^2}{2} - \langle x , a \rangle - b - \left(\frac{|x-m|^2}{2} - \langle x-m ,a \rangle - b\right) 
	\nonumber \\
	& = & q + \langle x,m \rangle - \frac{|m|^2}{2} - \langle m,a \rangle.
	\nonumber \\
	& = & q_0 + \langle x,a+m \rangle + b - \frac{|m|^2}{2} - \langle m,a \rangle
	\nonumber
\end{eqnarray}
In particular $\Pi$ acts on $\Omega^*=K^*/\R$ by translations and $M^*$ is isomorphic (as a smooth manifold) to the standard torus $\R^n/\mathbb{Z}^n$.
\end{example}

The manifolds in the above two examples are special, however as the example below illustrates our definitions work out also for non-special manifolds.

\begin{example}
Consider the action by $\mathbb{Z}$ on the positive real numbers generated by $y\mapsto 2y$. Let $M = \R_+/2^{\mathbb{Z}}$ be the quotient and $\phi$ and $L$ be the data defined by the metric $dy\otimes dy/y^2$ on $M$, in other words $\Phi_{q_0} = -\log(y)$ for some $q_0\in \Gamma(\Omega,K)$. We see that $-\log y - \langle y,a \rangle - b$ is bounded from below and proper if and only if $a<0$. This means $\Omega^*$ consists of all $q=q_0 + \langle y,a \rangle + b$ where $a<0$. Given a deck transformation $\gamma_m: y\mapsto 2^m y$ we calculate $\gamma. q$ by
\begin{eqnarray}
	\gamma_m.q & = & \Phi_q - \Phi_q\circ \gamma^{-1} \nonumber \\
	& = &  q -\log y -\langle a,y \rangle - b -  (-\log 2^{-m} y - \langle a,2^{-m} y \rangle - b) 
	\nonumber \\
	& = &  q_0 + \langle y,2^{-m} a \rangle + b - m\log 2. 
	\nonumber
\end{eqnarray}
In particular, if we identify an element $q=q_0 + \langle y,a \rangle + b$ in $\Omega^*$ with $a<0$ then the action by $\Pi$ on $\Omega^*$ is described by $\gamma_m:a\mapsto 2^{-m} a$ and $M^*\cong\R_-/2^\mathbb{Z}$.
\end{example}

\subsection{Legendre Transform}
\label{sec:legendre}
%%% Definition with pairing
%\begin{definition}
%Let $(M,L,\phi)$ be a Hessian space. We define the Legendre transform of $\phi$ as the section of $L^*\rightarrow M^*$ given by
%\begin{equation}
%\phi^*(p) = \sup_{x\in M} [x,p]-\phi(x).
%\label{eq:legendre_transform}
%\end{equation}
%\end{definition}

We begin by defining the Legendre transform of a section of $L\rightarrow M$ as a section of $-K^*\rightarrow \Omega^*$. In Proposition~\ref{prop:legequiv} we show that it is equivariant, in other words that it descends to a section of $-L^*\rightarrow M^*$. 

\begin{definition}[Legendre transform on the cover]
	Let $(M,L)$ be a Hessian manifold. Then the Legendre transform of a continuous section $\phi$ of $L$ is the convex section of the affine $\R$-bundle $-K^* \rightarrow \Omega^*$ defined by

\begin{equation}
	\phi^*(p) := -q + \sup_{x\in M} q(x) - \phi(x) = -q + \sup_{x\in M} -\Phi_q(x),
\label{eq:legendre_simply}
\end{equation}

where $q\in K^*$ is any point in the fiber over $p\in \Omega^*$.
\end{definition}

To see that the Legendre transform is well-defined, we must verify that it is independent of choice of $q$, but this follows immediately since any other choice can be written as $q' = q + m$ for some $m\in \mathbb{R}$, and thus

%\begin{equation}
$$
	-q' + \sup_{x\in \Omega} q'(x) - q^*\phi(x)  = -q - m + \sup_{x\in \Omega} q(x) + m - \phi(x) = \phi^*(p).
$$
%\label{eq:defleg}
%\end{equation}

Also note that the $\sup$ in \eqref{eq:legendre_simply} is always attained, since $p\in \Omega^*$ means that $\Phi_q$ is bounded from below and proper. 

\begin{remark}
	Note that over a simply connected manifold $\Omega$, any point $q\in \Gamma(\Omega,K)$ defines a global affine trivialization of the affine $\R$-bundle $K\rightarrow \Omega$. Since $\phi = \phi - q + q$, the representation of $\phi$ as function in this trivialization is simply $\Phi_q = \phi - q$. Thus the Legendre transform over a simply connected manifold can be viewed as taking the $\sup$ in different trivializations of the affine $\R$-bundle $K\rightarrow \Omega$. 
\end{remark}
\begin{remark}
	As in Remark~\ref{rem:identification} fix $x_0\in \Omega$, a basis of $T_{x_0}\Omega$ and $q_0\in K^*$. For each $p\in \Omega^*$, let $L(p)$ be the unique element $q$ in the fiber above $p$ such that $q_0(x_0)=q(x_0)$. Then $L$ defines an affine section of $K^*$. Moreover, using the idetification of $\Omega$ with a subset of $\R^n$, $L(p)-q_0$ may be identified with an element in $(\R^n)^*$.	Letting $q=L(p)$ and plugging this into \eqref{eq:legendre_simply} gives
\begin{equation}
	\phi^*(p) = L(p) + \sup_{x\in \Omega} \Phi_{L(p)} = L(p) + \sup_{x\in \Omega} (q-q) - \Phi_{q_0} = L(p) + \Phi_{q_0}^*(L(p)-q_0) 
\end{equation}
where $\Phi_p^*$ denotes the Legendre transform of $\Phi_{q_0}$, seen as a bona fide convex function on a convex domain in $\mathbb{R}^n$. We conclude that 
$$ \phi^*-L = \Phi_{q_0}^*. $$
\label{rem:leg-bi}
\end{remark}

\begin{proposition}
\label{prop:legequiv}
	The Legendre transform $\phi^*$ is $\Pi$-equivariant, i.e. $\phi^*(\gamma.p) = \gamma.\phi^*(p)$ for all $\gamma \in \Pi$.
\end{proposition}

\begin{proof}
	Fix $\gamma \in \Pi$. By Lemma~\ref{lemma:action1}, we have $\Phi_{\gamma.q} = \Phi_q \circ \gamma^{-1}$. Thus
	
\begin{equation}
	\phi^*(\gamma.p) = -\gamma.q + \sup_{\Omega} -\Phi_q\circ \gamma^{-1} = -\gamma.q + \sup_{\Omega} -\Phi_q =
	-\gamma.(q + \sup_{\Omega} \Phi_q) = \gamma.\phi^*(p).
\end{equation}
\end{proof}

We will now define a map $d\phi:\Omega\rightarrow \Omega^*$. It turns out that if $M$ is a compact manifold, and $\phi$ is smooth, then this map is a diffeomorphism. The map will also be equivariant. This will guarantee that the action of $\Pi$ on $\Omega^*$ induces a smooth quotient manifold $\Omega^*/\Pi=M^*$. Moreover, the map will also provide a diffeomorphism between $M$ and $M^*$ proving that they are equivalent as smooth manifolds.  
\begin{definition}\label{def:differential}
Let $(M,L,\phi)$ be a Hessian manifold and $x\in \Omega$. Locally there is a unique affine section tangent to $\pi^*\phi$ at $x$. By Proposition \ref{pr:extension_of_sections} this extends to a global affine section $q\in K^*$ (thus satisfying $\Phi_q(x)=d\Phi_q(x)=0$). We define $d(\pi^* \phi)(x)$ as the image of $q$ in $\Omega^*$.
\end{definition}

\begin{lemma}
	Let $(M,L,\phi)$ be a compact Hessian manifold. Then $d (\pi^*\phi): \Omega \rightarrow \Omega^*$ is a diffeomorphism.
\end{lemma}
\begin{proof}
	As in Remark~\ref{rem:identification}, we may identify $d(\pi^* \phi)$ with the map $d\Phi_{q_0}$. But since $\Phi_{q_0}$ is smooth and strictly convex, this yields an diffeomorphism.
\end{proof}
Moreover, we have
\begin{lemma}
\label{lemma:equivariant}
The map $d (\pi^*\phi): \Omega \rightarrow \Omega^*$ is equivariant. In other words
$$ d (q^*\phi)(\gamma(x)) = \gamma.d(q^*\phi)(x). $$
\end{lemma}
\begin{proof}
By Lemma~\ref{lemma:action1} $\Phi_{\gamma.q}=\Phi_q\circ\gamma^{-1}$. Hence, if $q$ is tangent to $\pi^* \phi$ at $x$ then $\gamma.q$ is tangent to $\pi^* \phi$ at $\gamma(x)$.
\end{proof}
\begin{theorem}\label{thm:diffeomorphism}
The quotient $M^*=\Omega/\Pi$ is an affine manifold and $d\phi:M\rightarrow M^*$ is a diffeomorphism.
\end{theorem}
\begin{proof}
By Lemma~\ref{lemma:equivariant} there is an equivariant diffeomorphism between $\Omega$ and $\Omega^*$. If the action by $\Pi$ on $\Omega$ induces a smooth quotient manifold, so does the action by $\Pi$ on $\Omega^*$. This means $M^*$ is an affine manifold. Moreover, we get the following commutative diagram
\begin{center}
\begin{tikzcd}
	\Omega \arrow{d}{\pi} \arrow{r}{d\pi^*\phi} & \Omega^* \arrow{d}{\tilde{\pi}} \\
	M \arrow{r}{d\phi} & M^*,
\end{tikzcd}
\end{center}
and since the top row is a diffeomorphism, so is the bottom. This means $M$ and $M^*$ are equivalent as smooth manifolds. 
\end{proof}

Using this diffeomorphism we also get an analogue of the real Legendre transform, in the sense that we can affinely identify the bidual $M^{**}$ with $M$.

\begin{lemma}
    Let $(M,L,\phi)$ be a compact Hessian manifold. Then the bidual $(M^{**},\phi^{**})$ and $(M,\phi)$ are isomorphic as Hessian manifolds.
\label{lem:bidual}
\end{lemma}

\begin{proof}
    Using the identification of Remark~\ref{rem:identification} twice, we have the following commutative diagram.
    
\begin{center}
\begin{tikzcd}
	\Omega \arrow{d}{i_1} \arrow{r}{d\pi^*\phi} & \Omega^* \arrow{d}{i_2}  \arrow{r}{d(\tilde{\pi}^* \phi^*)} &  \Omega^{**} \arrow{d}{i_3}\\
	i_1(\Omega) \arrow{r}{d\Phi_{q_0}} &           i_2(\Omega^*) \arrow{r}{d\Phi_{z_0}}   & i_3(\Omega^{**}),
\end{tikzcd}
\end{center}
where $z_0$ is some choice of affine section $z_0 \in \Gamma(\Omega^*,K^*)$. But taking $z_0$ as in Remark\ ~\ref{rem:leg-bi}, we have that $\Phi_{z_0} = (\Phi_{q_0})^*$. By standard properties of smooth strictly convex functions, we have that $d\Phi_{q_0}^* \circ d\Phi_{q_0} = Id$. But this shows that the identity map from $i_1(\Omega) \to i_3(\Omega^{**})$ is equivariant, and hence $M$ and $M^*$ are equivalent as affine manifolds. Further, since $\Phi_{q_0}^{**} = \Phi_{q_0}$ as convex functions, the equivalence indeed holds also in the Hessian category.
\end{proof}

 Note that by the above identification with the classical Legendre transform on the cover $\Omega$, we immediately inherit several properties from the corresponding properties of the Legendre transform in $\mathbb{R}^n$.
 In particular the above identification yields an identification of the bidual $M^{**} = M$. By taking the double Legendre transform $\Phi_{q_0}^{**}$ as a real function, we get a $\Pi$ invariant convex function on $\Omega$, which descends to a Hessian metric on $M$, and $\Phi_{q_0}^{**} = \Phi_{q_0}$. Furthermore, this construction is valid for any continuous section $s$, and hence we may define a projection operator taking continuous sections of $L\to M$ to convex sections of $L\to M$. By slight abuse of notation (i.e., identifying the bidual $M^{**} = M$, see Lemma \ref{lem:bidual}), we denote this projection by double Legendre transformation, and the following proposition follows.

\begin{proposition}
	Let $(M,L)$ be a compact Hessian manifold and $\phi$ a convex section of $L$. Then
\begin{align}
	\phi^{**}(x) &= \sup_{q \in \Gamma(\Omega,K), \Phi_q \geq 0} q(x) \\
	\phi^{**}(x) &= \phi(x)	
	\label{eq:sup_formula}
\end{align}

where on the right hand side we have identified $x$ with an arbitrary point over $x$ in the cover. Furthermore, for any continuous section of $L$, we have that

\begin{equation}
	s^{**} \leq s.
\end{equation}

pointwise.
\end{proposition}

Moreover, by standard properties of convex functions, for any convex (not necessarily strictly convex) section $\phi$, $d\Phi_{q_0}$ has an inverse defined almost everywhere, namely $d(\Phi_{q_0}^*)$. This means that, under the identification above $d(\phi^*)$ is an inverse of $d\phi$ defined almost everywhere on $M^*$. We will denote this map $T_\phi$. Moreover, by standard properties for convex functions, for any continuous $\Pi$-invariant function $v$ (see for example Lemma 2.7 in \cite{berman2013real}) 
$$ \frac{d}{dt}(\Phi_{q_0}+tv)^* = -v\circ (d\Phi_{q_0}^*). $$
It follows that
\begin{equation} \frac{d}{dt}(\phi+tv)^*(p) = -v(T_\phi(p)). \label{eq:legendre_variation} \end{equation}

We end this section with the following definition.
\begin{definition}
\label{def:numa}
Let $(M,L)$ be a Hessian manifold and $\nu$ a probability measure on $M^*$. We define the $\nu$-Monge-Amp\`ere measure of a convex section $\phi$ in $L$ as 
$$ \MA_\nu(\phi) = (T_\phi)_* \nu. $$
\end{definition}
\begin{remark}
It is interesting to note that there is no apparent complex geometric analogue of the $\nu$-Monge-Amp\`ere unless in the case when $M$ is special and $\nu$ is the unique parallel probability measure, in which case $\MA_\nu$ reduces to the standard Monge-Amp\`ere operator considered in \cite{chengyau1980}.
\end{remark}

\section{Solvability of Monge-Amp\`ere equations}
\label{sec:ma}
We are now ready to give proofs the Theorems~\ref{thm:main} and \ref{thm:intro-geom}. We begin by 
\begin{definition}
	For a Hessian manifold $(M,L,\phi_0)$, the \emph{affine Kantorovich functional} is the functional $F: C^0(M) \rightarrow \mathbb{R}$ defined by

\begin{equation}
	F(u) = \int_{M} u d\mu + \int_{M^*} \left[(u+\phi_0)^* - \phi_0^* \right]d\nu
\end{equation}

By abuse of notation, if $\phi$ is a convex section of $L$, we also write
\begin{equation}
	F(\phi) = \int_{M} (\phi - \phi_0) d\mu + \int_{M^*} (\phi^* - \phi_0 ^*) d\nu 
	\label{eq:kantorovich}
\end{equation}
\end{definition}
\begin{remark}
Note that \eqref{eq:kantorovich} only depends on $\phi_0$ up to a constant. In particular the minimizers of \eqref{eq:kantorovich} are independent of $\phi_0$. We stress that this is not the classical Kantorovich functional induced by the Riemannian metric $\nabla d\phi_0$. Rather, it is determined by the affine structure on $M$ together with $L$.  
\end{remark}

\begin{proposition}
	Let $(M,L,\phi_0)$ be a compact Hessian manifold. Let $\mu$ and $\nu$ be probability measures on $M$ and $M^*$ respectively. Then $F$ admits a convex minimizer. If $\nu$ is absolutely continuous with full supprt and if $\phi_0$ and $\phi_1$ are minimizers of $F$, then $\phi_1-\phi_0$ is constant. If, in addition, $\mu$ and $\nu$ are absolutely continuous with non-degenerate $C^{k,\alpha}$ densities for some $k\in \mathbb{N}$ and $\alpha\in (0,1)$, then any minimizer is in $\phi \in C^{k+2,\alpha}$.
\label{thm:monge-ampere}
\end{proposition}
%Given a convex section, $\phi$, of $L$ and $q\in K^*$ we may consider $\Phi_q$. Since this is a convex function it is differentiable almost everywhere. This means $d\phi:M\rightarrow M^*$ is defined almost everywhere on $M$. By convexity of $\phi$ we have, for each $x_0$ where $p_0=d\phi(x_0)$ is defined
%$$ \phi^*(p_0) = \sup [x,p_0] - \phi(x) = q_0 + \sup_{x\in M} p_0(x) - \phi(x) = q_0 +p(x) \phi(x).  $$
%Moreover, by standard properties of convex functions the derivative of $\Phi_q$ is invertible almost everywhere. This means there is a map $T_\phi:M^*\rightarrow M$ (the partially defined inverse of $d\phi$) defined almost everywhere with the property that whenever $x=T_\phi(p)$ is defined
%$$ \phi^*(p_0) = [x,p] -\phi(x) $$
%for a unique $x\in M$, namely $T_\phi(x)$. The map $T_\phi$ is the optimal transport map from the theory of optimal transport. We have the following
Before we prove this we will explain how it implies Theorems~\ref{thm:main} and \ref{thm:intro-geom}. The main point is given by the following characterization of the minimizers of $F$. 
\begin{proposition}
\label{thm:weak_solutions}
	Let $(M,L,\phi_0)$ be a compact Hessian manifold, and let $\mu$ and $\nu$ be probability measures on $M$ and $M^*$. Assume $\nu$ is absolutely continuous and let $\phi$ be a convex minimizer of \eqref{eq:kantorovich}. Then 
$$ (T_\phi)_* \nu = \mu $$
where $T_\phi$ is the map defined at the the end of the previous section.  
\end{proposition}
%This theorem will follow from the following
%\begin{lemma}
%	Let $v$ be a continuous function on $M$, $p\in M^*$ and $x$ is the unique point in $M$ such that 
%$$ \phi^*(p) = [x,p]-\phi(x). $$
%Then 
%$$ \frac{d}{dt} (\phi+tv)^*(p) = v(x). $$
%\end{lemma}
%\begin{proof}  
%	This follows from the general fact that if $G$ is a proper upper semi-continuous function, $v$ is a bounded continuous function and $\sup G$ is attained at a unique point $x$, then 
%$$ \sup G-tv $$
%is differentiable in $t$ and 
%$$ \frac{d}{dt} \sup G-tv = -v(x). $$
%See for example the appendix of \cite{BermanBerndtsson2012} for a proof of this.
%\end{proof}
\begin{proof}
	Let $v$ be a continuous function on $M$. First of all, we claim that
\begin{equation}
\sup_{p\in M^*} |(\phi+tv)^*(p) - \phi^*(p)| \leq \sup_{x\in M} |tv(x)|.
\label{eq:legendre_supnorm}
\end{equation}
We defer the proof of this claim to the end of the proof of the existence part Proposition~\ref{thm:weak_solutions}. The dominated convergence theorem and \eqref{eq:legendre_supnorm} then gives

\begin{equation} 
	\frac{d}{dt}F(\phi+tv) = \int_M v \mu + \int_{M*} \frac{d}{dt} (\phi+tv)^* d\nu.
	\label{eq:gateaux_derivative}
\end{equation}
By \eqref{eq:legendre_variation} and since $\nu$ is absolutely continuous we have that $\nu$-almost everywhere 
$$\frac{d}{dt} (\phi+tv)^*(p)=-v(T_\phi(p)).$$
Applying this to the second integral above and performing the change of variables formula $x=T_\phi(p)$ we get 
$$ \frac{d}{dt}F(\phi+tv) = \int_M v (\mu-(T_\phi) _* \nu). $$
Since $\phi$ is a minimizer of $F$ this has to vanish for any $v$, hence $\mu-(T_\phi) _* \nu=0$.
\end{proof}

Combining Propositions~\ref{thm:monge-ampere} and \ref{thm:weak_solutions}, Theorems~\ref{thm:main} and \ref{thm:intro-geom} follow by the following arguments. Note that since $\pi$, $\tilde{\pi}$ are covering maps, we may consider the pullbacks $\pi^* \mu$, $\tilde{\pi}^* \nu$ as invariant measures on $\Omega, \Omega^*$, and moreover any invariant measures on $\Omega,\Omega^*$ arise in this way. Moreover, by definition, $T_\phi:M^* \rightarrow M$ is induced by the (equivariant) partially defined inverse of $d\Phi_{q_0}\Omega \rightarrow \Omega^*$. It then follows that 
\begin{equation}
\label{eq:cover_push}
\pi^*\mu = (d\Phi_{q_0})^{-1}_* \pi^* \nu 
\end{equation}
if and only if 
$$ \mu = (T_\phi)_* \nu. $$
Under the assumption that $\nu$ is absolutely continuous, \eqref{eq:cover_push} is equivalent to $\Phi_{q_0}$ being an Alexandrov solution to \eqref{eq:m-a_equation} (see for example Lemma 4.2 in \cite{villani2003topics}). 
This means Theorem~\ref{thm:main} is a direct consequence of Theorem~\ref{thm:monge-ampere} and Theorem~\ref{thm:weak_solutions}.

We now turn to the proof of Theorem~\ref{thm:monge-ampere}. To establish existence of minimizers we will need a $C^0$-estimate and a Lipschitz bound on (normalized) convex sections of $L$, which together imply the following theorem, using Arzela-Ascoli.

\thmCompact*

\begin{proposition}[Uniform $C^0$ estimate]
	Let $(M,L,\phi_0)$ be a compact Hessian manifold. Then any $\phi$ in the K\"{a}hler class of $\phi_0$, normalized such that $\sup \phi - \phi_0 = 0$ satisfies $|\phi - \phi_0| \leq C$ for some constant depending only on $\phi_0$.
\label{prop:c0-estimate}
\end{proposition}

\begin{proof}
	Fix $\phi$, and let $u = \phi-\phi_0 \in C^0(M)$. Being the difference of two convex sections, $u$ has a Hessian in the Alexandrov sense. Fix $K \subset \mathbb{R}^n$ as a relatively compact convex set containing a fundamental domain of $M$ in $\Omega$. Then the affine curve $x_t = (1-t) x_0 + t x_1 \in K$, where we identify $x_0,x_1$ with any lift to $K$. Letting $f(t) = u(x_t)$ we have $f'(0) = f'(1) = 0$, and thus, letting $\nabla^2$ denote the Alexandrov Hessian given by the embedding $\Omega \subseteq \mathbb{R}^n$ endowed with the Euclidean metric we have
\begin{equation}
\begin{split}
	f(t) &= f(0) + \int_0^t \left( \int_0^s f''(l) dl \right) ds \\
&= f(0) + \int_0^t \left( \int_0^s  \langle x'_l , \nabla^2 u|_{x_l} x'_l\rangle dl \right) ds \\
&\geq f(0) - \int_0^t \left( \int_0^s  \langle x'_l , \nabla^2 \phi_0|_{x_l} x'_l\rangle dl \right) ds \\
&\geq f(0) - Ct^2,
\end{split}
\end{equation}
where the first inequality follows since $\nabla^2 u = \nabla^2 \phi - \nabla^2 \phi_0$ and $\nabla^2 \phi \geq 0$ by convexity. The constant $C$ depends only on $\phi_0$ and the (bounded) diameter of $K$. For $t=1$ this yields that $\sup u - \inf u \leq C$, and the proposition follows. 
\end{proof}
Virtually the same proof can be used to give a locally uniform Lipschitz bound.

\begin{proposition}
	Let $(M,L,\phi_0)$ be a compact Hessian manifold, and let $\phi$ be a convex section of $L$. Then for $u := \phi - \phi_0$ we have $\|u\|_{Lip} \leq C$ on any compact $K\subset \Omega$, for some constant depending only on $(M,L,\phi_0)$ and $K$, where $\|.\|_{Lip}$ denotes the Lipschitz constant with respect to the Euclidean metric on $\R^n$. 
\label{prop:lip-estimate}
\end{proposition}
\begin{proof}
	Fix a compact set $K \subset \Omega$. We may without loss of generality assume that $K$ is a convex set containing a fundamental domain. For any $x\in K$ there is a Euclidean open ball $B(x,r)$ of radius $r$ such that $B(x,r) \subset \Omega$, and by compactness a finite number of such balls cover $K$, and we let $U=\bigcup_{i=1}^N B(x_i,r_i)$ be their union. It follows that $x + tv\in U$ for all $x\in K$, $t\in [-2,2]$ and $\|v\| \leq \delta := \min r_i/3 > 0$. Now fix $x\in K$ arbitrarily and fix $v$ arbitrarily such that $\|v\| \leq \delta$. Consider the function

\begin{equation}
	f(t) := u(x + tv).
\end{equation}
as a function of $t$, twice differentiable in the Alexandrov sense, and defined on some open interval V such that $[0,1] \subset V$. Now assume that $du_x(v) = A > 0$ for some $A$, or equivalently, $f'(0) = A$. We then have

\begin{equation}
\begin{split}
	f(t) &= f(0) + \int_0^t \left[f'(0) + \int_0^s f''(0) d\tau\right] ds \\
&= f(0) + tA + \int_0^t\int_0^s \langle v \nabla^2 u |v\rangle d\tau ds \\
&\geq f(0) + tA + \int_0^t \int_0^s -C_{\phi_0} d\tau ds \\
&= f(0) + tA - \frac{t^2 C_{\phi_0}}2.
\end{split}
\end{equation}
for some constant $C_{\phi_0}$ depending only on $\phi_0$ and $\delta$. Then setting $t=1$ we get 

\begin{equation}
	A \leq f(1) - f(0) + C_{\phi_0} \leq \sup u - \inf u + C_{\phi_0},
\end{equation}
and by Proposition \ref{prop:c0-estimate} we get an uniform upper bound on $A$. Replacing $v$ by $-v$ yields also a uniform lower bound, which then gives the desired bound on $\|u\|_{Lip}$.
\end{proof}

Using these a priori estimates the existence of a minimizer can be established.

\begin{proof}[Proof of the existence part of Proposition~\ref{thm:weak_solutions}]
	Let $\phi_k$ be an infimizing sequence of $F$, and define $u_k := \phi_k - \phi_0 \in C^0(M)$. First we note that the functional $F$ is invariant under the map $\phi \mapsto \phi + C$, and hence we may assume that the sequences are normalized such that $\int_M (\phi_k - \phi_0 )d\mu = \int_M u_k d\mu = 0$. Second we note that since $F(\phi^{**}) \leq F(\phi)$, we may assume that $\phi_i$ lie in the Kahler class of $\phi_0$. Then, since $u_k \in C_0(M)$ it follows that $\sup_X u_k \geq 0 \geq \inf_X u_k$, and hence by Proposition \ref{prop:c0-estimate} $\|u\|_{C_0(M)}$ is uniformly bounded. Furthermore, $\|u\|_{Lip}$ is uniformly bounded by Proposition \ref{prop:lip-estimate}. By the Arzela-Ascoli theorem we can thus extract a subsequence converging as $u_k \rightarrow u$ in $C_0(M)$, and thus also convergence $\phi_k \rightarrow \phi$. To show that $\phi$ is indeed a minimizer of $F$ it remains to show that $F$ is continuous as a map $C_0(M) \rightarrow \mathbb{R}$. To show this it suffices to show that $\phi^* = \lim (\phi_k)^*$. But this follows from the general claim that $|\inf f - \inf g| \leq \sup |f - g|$, since $|\phi^*(p) - \phi_k ^*(p)| = |\inf_{x\in \Omega} \Phi_q(x) - \inf_{x\in \Omega} \Phi_{k,q}|$. To show the claim, assume that $\inf f \leq \inf g$, let $x_\epsilon$ be such that $\inf f \geq f(x_\epsilon) - \epsilon$. Then we have that
	
	\begin{equation}
	    -|\inf f - \inf g | = \inf f - \inf g \geq f(x_\epsilon) - \epsilon  - g(x_\epsilon) \geq -\epsilon - \sup |f-g|.
	\end{equation}
	Letting $\epsilon \rightarrow 0$ yields the claim.
\end{proof}

%The fact that the minimizer is a Lipschitz function allows us to bootstrap to get $C^{2,\alpha}$ regularity for any $\alpha \in (0,1)$ by using Caffarelli's regularity results (see, e.g., \cite{villani2008optimal} for convex domains in $\mathbb{R}^n$

\begin{proof}[Proof of regularity part of Proposition~\ref{thm:monge-ampere}]
Fix a point $x\in \Omega$, a point $p\in \Omega^*$ and a small open ball $B(x,r)\ni x$. Then, since $\phi$ solves a Monge-Ampere equation it follows that $d\Phi_p: B(x,r) \rightarrow d\Phi_p(B)$ is a Brenier map for an optimal transportation of retrictions of $\mu$ and $\nu$. By Caffarelli's regularity theory \cite{villani2008optimal}[Thm 4.14], since $d\Phi_p(B)$ is a convex domain, it then follows that we have that $\Phi_p \in C^{2,\alpha}(B(x,r))$, and thus also that $\pi^*u = \Phi_p - \Phi_{0,p} \in C^{2,\alpha}(B(x,r))$. But fixing a relatively compact set $K$, covering $\bar{K}$ with $B(x,r/2)$ and passing to a finite subcover yields that $u \in C^{2,\alpha}(K)$. The same argument yields the $C^{k+2,\alpha}$ result.
\end{proof}

Uniqueness of minimizers follows from a convexity argument.

\begin{proof}[Proof of uniqueness part of Proposition~\ref{thm:monge-ampere}]
	Assume that there are two minimizers $\psi_0,\psi_1$, both normalized such that $\int (\psi_i - \phi_0) \mu =0$, and let $\psi_t = (1-t)\psi_0 + t\psi_1$. Then

\begin{equation}
\begin{split}
	\inf \Psi_{t,p} = \inf \left[(1-t) (q^*\psi_0 - p) + t(q^*\psi_1 - p) \right] \geq (1-t) \inf \Psi_{0,p} + t \inf \Psi_{1,p}.
\end{split}
\end{equation}
and hence $\psi^*_t \leq (1-t)\psi^*_0 + (1-t)\psi_1^*$ holds pointwise. It follows that $F(\psi_t) \leq F(\psi_0)$. Now assume that $\psi^*_t(p) < \psi^*_0(p)$ in some point $p$. By continuity this then holds also for some open set $U \in M^*$. But using the pointwise inequality on $M\setminus U$ and strict inequality on $U$ we get that $\int_{M^*}( \psi^*_t - \phi^*_0 ) d\nu < \int_{M^*}( \psi^*_0 - \phi^*_0 )d\nu$, contradicting the minimality of $\psi_0$. Hence $\psi^*_t = \psi^*_0$ for all $t$, and uniqueness follows.
\end{proof}

Using the uniqueness result we are also able to show continuity of the inverse Monge-Ampére operator.

\begin{theorem}
	Let $(M,L,\phi_0)$ be a compact Hessian manifold, and let $\nu$ be a fixed absolutely continuous probability measure on $M^*$, with full support. Then if $\mu_i \rightarrow \mu$ are probability measures converging in the weak$^*$ topology, the solutions $\phi_i \rightarrow \phi$ of $\MA_\nu \phi_i = \mu_i$, normalized such that $\int_M ( \phi^*_i - \phi^*_0 )\mu = 0$, converge in the $C^0$-topology, where $\MA_\nu \phi = \mu$.
\label{thm:inv-ma-cont}
\end{theorem}

\begin{proof}
    We claim that Theorem~\ref{thm:compactness} yields that up to subsequence $\phi_i \rightarrow \bar{\phi}$ in the $C^0$ topology. Indeed, as in the existence proof of Proposition~\ref{thm:monge-ampere}, we have that $\phi_i^*$ has a converging subsequence, and the claim then follows from the continuity of the Legendre transform. Furthermore, note that in fact $\int_{M^*} (\bar{\phi}^* - \phi_0 ^*) d\nu = 0$. 
	
	Let $F_i$, $F$ denote the Kantorovich functionals corrensponding to $\mu_i, \mu$, and let $\phi$ be the solution to $\MA_\nu \phi = \mu$, normalized such that $\int_{M^*} (\phi^* - \phi_0^*) d\nu = 0$. Then we by minimality of $\phi_i$ for $F_i$ have that

\begin{equation}
\begin{split} 
	F_i(\phi_i) 	&\leq F_i(\phi) = \int_M (\phi- \phi_0)d\mu_i \\
			&= \int_M (\phi-\phi_0) d\mu + \int_M (\phi-\phi_0)(d\mu_i - d\mu).
\end{split}
\end{equation}
Since $\mu_i \rightarrow \mu$ and $\phi-\phi_0$ is bounded and continuous, taking limits we obtain $\limsup F_i(\phi_i) \leq F(\phi)$. On the other hand we have

\begin{equation}
\begin{split}
	F_i(\phi_i) &= \int_M (\phi_i -\phi_0)d\mu_i \\
		    &= \int_M(\bar{\phi} - \phi_0)d\mu + \int_M(\phi_i - \bar{\phi})d\mu_i + \int_M(\bar{\phi} - \phi_0) (d\mu_i - d\mu) \\
		    &= F(\bar{\phi}) + \int_M(\phi_i - \bar{\phi})d\mu_i + \int_M(\bar{\phi} - \phi_0) (d\mu_i - d\mu).
\end{split}
\end{equation}
Since $\phi_i \rightarrow \bar {\phi}$ and $\mu_i$ is of mass $1$, we have that $\left|\int_M(\phi_i - \bar{\phi})d\mu_i\right| \leq \sup_M |\phi_i-\bar{\phi}|  \rightarrow 0$, and by weak-$*$ convergence we have that $\int_M(\bar{\phi} - \phi_0) (d\mu_i - d\mu) \rightarrow 0$. Taking limits we thus obtain that $F(\bar{\phi}) = \lim F_i(\phi_i) \leq F(\phi)$, which shows that $\bar{\phi}$ is a minimizer of $F$. By the uniqueness part of Proposition~\ref{thm:monge-ampere} it follows that $\bar{\phi} = \phi$, and consequently $\phi_i \rightarrow \phi$
\end{proof}

\section{The Pairing and Optimal Transport}
\label{sec:pairing}
Let $M_1$ and $M_2$ be two affine manifolds. Consider their product $M_1\times M_2$ and let $q_1$ and $q_2$ be the projections of $M_1\times M_2$ onto $M_1$ and $M_2$ respectively. Assume $L_1$ and $L_2$ are affine $\R$-bundles over $M_1$ and $M_2$ respectively. Then there is a natural affine $\R$-bundle over $M_1\times M_2$ given by
$$ 
L\boxplus L^* = q_1^* L_1 + q_2^* L_2. 
$$
Given a Hessian manifold $(M,L,\phi)$ we will show that $L\boxplus -L^*$ has a canonical section. We will use the notation $[\cdot,\cdot]$ for this section and it will play the same role as the standard pairing between $\R^n$ and $(\R^n)^*$ in the classical Legendre transform. The definition will be given in terms of a section in $K\boxplus -K^*$. We will then show that this section defines a section in $L\boxplus -L^*$. Indeed, the actions of $\Pi$ on $K$ and $K^*$ defines an action by $\Pi\times \Pi$ on $K\times -K^*$ given by
$$ 
(\gamma_1,\gamma_2) (y-q) = \gamma_1(y) - \gamma_2.q
$$
and we may recover $L\boxplus -L^*$ as the quotient $K\boxplus -K^*/\Pi\times \Pi$.

%Let $K_x$ the fiber of $K$ above $x\in \Omega$ and $K^*_p$ be the fiber of $K^*$ over $p\in \Omega^*$. An element in $K\boxplus K^*$ in the fiber above $(x,p)\in M\times M^*$ is given by the sum $y+q$ where $y\in K_x$ and $q\in K^*_p$ or, more precisely, the pair $(y,q)\in K_x\times K^*_p$ where we identify pairs $(y+C,q)$ with pairs $(y,q+C)$ for all $C\in \R$.

\begin{definition}
Let $(M,L,\phi)$ be a Hessian manifold and $K\rightarrow \Omega$ and $K^*\rightarrow \Omega^*$ be the associated objects defined in the previous section. Given $(x,p)\in \Omega \times \Omega^*$, let $q$ be a point in the fiber of $K^*$ over $p$. We define
$$
[x,p] = \sup_{\gamma\in \Pi} \gamma.q(x)-q. 
$$
\end{definition}

\begin{remark} 
To see that this is well-defined we must verify that it is independent of the choice of $q$. But this follows immediately since any other choice can be written as $q' = q + C$ for some $C\in \mathbb{R}$ and thus
$ \gamma.q'(x)-q' = \gamma.q(x)-q. $
\end{remark}

\begin{lemma}
The pairing $[\cdot,\cdot]$ descends to a section of $L\boxplus -L^*$. 
\end{lemma}
\begin{proof}
We need to prove that $[\cdot,\cdot]$ is equivariant, in other words that 
$$ [\gamma_1(x),\gamma_2.p] = (\gamma_1,\gamma_2)[x,p] $$
for all $\gamma_1,\gamma_2\in \Pi$, $x\in M$ and $p\in M^*$. Now, 
\begin{eqnarray}
[\gamma_1(x),p] & = & \sup_{\gamma\in \Pi} \gamma.q(\gamma_1(x))-q
\nonumber \\
& = & \sup_{\gamma\in \Pi} \gamma_1(\gamma_1^{-1}\gamma.q(x))-q
\nonumber \\
& = & \sup_{\gamma\in \Pi} \gamma_1(\gamma.q(x)))-q
\nonumber \\
& = & (\gamma_1,id)[x,p].
\nonumber
\end{eqnarray} 
where the second equality follows from
$$ \gamma.q(\gamma_1(x)) = \gamma\circ q \circ \gamma^{-1}\circ \gamma_1(x) = \gamma_1 \circ (\gamma_1^{-1}\circ\gamma)^{-1} \circ q \circ (\gamma_1^{-1} \circ \gamma)^{-1}(x) = \gamma_1(\gamma_1^{-1}\gamma.q(x)) $$
and the third equality follows from the substitution of $\gamma$ by $\gamma_1^{-1}\gamma$.  
Moreover, 
$$
[x,\gamma_2.p] = \sup_{\gamma\in \Pi} \gamma.\gamma_2.q(x) - \gamma_2.q
= \sup_{\gamma\in \Pi} \gamma.q(x)-\gamma_2.q = (id, \gamma_2)[x,p].
$$ 
where the second equality is given by substituting $\gamma$ by $\gamma \gamma_2$. 
This proves the lemma.
\end{proof}

\begin{lemma}
\label{lemma:legdefpairing}
Assume $(M,L)$ is a compact Hessian manifold and $\phi$ is a continuous section of $L$, then
\begin{equation} \phi^*(p) = \sup_{x\in M} [x,p]-\phi(x). \label{eq:legdefpairing} \end{equation}
Moreover, $d\phi$ is defined at a point $x\in M$ and $d\phi(x)=p$ if and only if $p$ is the unique point in $M^*$ such that
$$ \phi^*(p)=[x,p]-\phi(x). $$
\end{lemma}
\begin{proof}
The right hand side of \eqref{eq:legdefpairing} is given by
\begin{eqnarray} 
\sup_{x\in M} \sup_{\gamma\in \Pi} \gamma.q(x) - q - \phi(x) 
& = & -q +\sup_{x\in M} \sup_{\gamma\in \Pi} -\Phi_{\gamma.q}(x) \nonumber \\
& = & -q +\sup_{x\in M} \sup_{\gamma\in \Pi} -\Phi_q(\gamma^{-1}(x)) \nonumber \\
& = & -q +\sup_{x\in \Omega} -\Phi_q(x) = \phi^*(p) \nonumber
\end{eqnarray}
where, as usual, $q$ is a point in the fiber above $p$.
To prove the second point, note that $d\phi(x)=p$ if and only if there is $\tilde x\in \Omega$ above $x$ and $q\in K^*$ above $p$ such that
$$ d\Phi_q(\tilde x) = 0. $$
By standard properties of convex functions this is true if and only if 
$$ \Phi_q^*(0) = - \Phi_q(\tilde x).  $$
Since $\Phi_q^*(0) \geq -\Phi_q(x') $ for any $x'\in \Omega$ we get
\begin{equation} \Phi_q^*(0) = - \sup_{\gamma\in \Pi} \Phi_q\circ\gamma^{-1}(\tilde x) \label{eq:diffargument} \end{equation}
Using the notation in Remark~\ref{rem:leg-bi} we have $\Phi_q^* = \phi^*+L$. Putting $q=L(p)$ we get, using that $\Phi_q\circ \gamma = \Phi_{\gamma.q}$, that \eqref{eq:diffargument} is equivalent to
\begin{eqnarray}
\phi^*(p) & = & -q +\sup_{\gamma\in \Pi} \Phi_{\gamma.q}(\tilde x) \nonumber \\
& = & \sup_{\gamma\in \Pi}-q + \gamma.q(\tilde x) - \phi(x) \nonumber \\
& = & [x,p]-\phi(x). \nonumber
\end{eqnarray}
\end{proof}

\subsection{Recap: Kantorovich Problem of Optimal Transport}
Let $X$ and $Y$ be topological manifolds, $\mu$ and $\nu$ (Borel) probability measures on $X$ and $Y$ respectively and $c$ be a real-valued continuous function on $X\times Y$. Then the associated problem of optimal transport is to minimize the quantity 
$$ I_c(\gamma) = \int_{X\times Y} c(x,y) \gamma $$
over all probability measures $\gamma$ on $X\times Y$ such that its first and second marginals coincide with $\mu$ and $\nu$ respectively. A probability measure $\gamma$ with the above property is called a transport plan. Under regularity assumptions (see \cite{villani2008optimal}) on $\mu$, $\nu$ and $c$, $I_c$ will admit a minimizer $\bar{\gamma}$ which is supported on the graph of a certain map $T:X\rightarrow Y$ called the optimal transport map. If this is the case then $\bar{\gamma}$ is determined by $T$ and 
$$ \bar{\gamma} = (\id\otimes T)_* \mu. $$ 
where $\id$ is the identity map on $X$. 
\begin{remark}
See the introductions of \cite{villani2003topics} and \cite{villani2008optimal} for very good heuristic interpretations of transport plans and transport maps.  
\end{remark}
\begin{remark}\label{rem:equivalentcostfunctions}
Assume two cost functions $c$ and $c'$ satisfy
\begin{equation}
\label{eq:equivalentcostfunctions}
c'(x,y) = c(x,y) + f(x) + g(y) \end{equation}
where $f\in L^1(\mu)$ and $g\in L^1(\nu)$ are functions on $X$ and $Y$ respectively. Then they determine the same optimal transport problem in the sense that
\begin{eqnarray} I_{c'}(\gamma) & = & \int_{X\times Y} c' \gamma = \nonumber \\
& = & \int c \gamma + \int_X f \gamma + \int_Y g \gamma \nonumber \\
& = & \int c \gamma + \int_X f \mu + \int_Y g \nu \nonumber \\
& = & I_{c} + C \nonumber
\end{eqnarray}
where $C$ is a constant independent of $\gamma$. In particular, $I_{c}$ and $I_{c'}$ have the same minimizers. Motivated by this we will say that two cost functions $c$ and $c'$ are equivalent if \eqref{eq:equivalentcostfunctions} holds for some integrable functions $f$ and $g$ on $X$ and $Y$ respectively.
\end{remark}
%\begin{remark}
%A heuristic interpretation of this is that a certain amount of material need to be moved from one configuration, described by $\mu$, to another configuration, described by $\nu$ and that $c(x,y)$ is the cost of moving one unit of material from $x$ to $y$. Then $\gamma$ is a plan on how to perform this in the sense that the amount of material to be moved from $A\subset X$ to $B\subset Y$ is given by $\gamma(A\times B)$. Then $I(\gamma)$ is the cost of carrying out this plan. Moreover, if $\gamma$ is supported on the graph of $T$, then the plan prescribes all material at a point $x$ to be sent to $T(x)$. 
%\end{remark}

Two important cases is worth mentioning. The first is when $X=Y$ is a Riemannian manifold and $c(x,y) = d^2(x,y)/2$ where $d$ is the distance function induced by the Riemannian metric. The other, which can in fact be seen as a special case of this, is when $X=\R^n$ and $Y=(\R^n)^*$ and $c(x,y)=-\langle x,y \rangle$, where $\langle \cdot,\cdot \rangle$ is the standard pairing between $X$ and $Y$. Now, let $d$ be the standard Riemannian metric on $\R^n$. This induces an isomorphism of $X$ and $Y$ and 
$$ d(x,y)^2/2 = |x-y|^2/2 = |x|^2/2 - \langle x,y \rangle + |y^2|/2. $$
In other words $-\langle \cdot,\cdot \rangle$ and $d(\cdot,\cdot)^2/2$ are equivalent as cost functions, as long as $\mu$ and $\nu$ have finite second moments. 

To see the relation between our setup and optimal transport we need to consider the Kantorovich dual of the problem of optimal transport. Let $f$ be a continuous function on $X$. Its $c$-transform is the function on $Y$ given by
$$ f^c(y) = \sup_{x\in X} -c(x,y) - f(x). $$
Moreover, if $x\in X$ satisfies
$$ f^c(y) = -c(x,y) - f(x) $$
for a unique $y\in Y$, then the $c$-differential of $f$, $d^c f$, is defined at $x$ and $d^cf(x) = y$. 

The dual formulation of the problem of optimal transport above is to minimize the quantity
$$ J(f) = \int_X f \mu + \int_Y f^c \nu. $$ 
over all continuous functions on $X$. Let $\Pi(\mu,\nu)$ be the set of transport plans. We have the following 
\begin{theorem}[See theorem 5.10 in \cite{villani2008optimal}]
\label{thm:villani}
Let $X$, $Y$, $\mu$, $\nu$ and $c$ be defined as above. Then, under certain regularity conditions (see \cite{villani2008optimal} for details)
\begin{equation}
\label{eq:duality}
\inf_\gamma I(\gamma) = -\inf_\phi I(\phi). 
\end{equation}
Moreover, both $I$ and $J$ admits minimizers and the minimizer of $I$ is supported on the graph of $T= d^c f$ were $f$ is the minimizer of $J$. 
\end{theorem}

\subsection{Affine $\R$-Bundles and Cost Functions}
\begin{definition}
Let $(M,L,\phi_0)$ be a compact Hessian manifold. We say that the associated cost function on $M\times M^*$ is
$$ c(x,y) = -[x,y]+\phi_0(x)+\phi_0^*(y). $$
\end{definition}
\begin{remark}
Let $\phi$ be another convex section of $L$. Then $\phi - \phi_0$ and $\phi^* - \phi_0^*$ are continuous functions on $M$ and $M^*$ respectively. Let $c'$ be the cost function induced by $(M,L,\phi)$. Then 
$$ c'(x,y) = c(x,y) - (\phi-\phi_0) - (\phi^*-\phi_0^*). $$
This means $(M,L,\phi_0)$ and $(M,L,\phi)$ determine equivalent cost functions (in the sense of Remark~\ref{rem:equivalentcostfunctions}). We conclude that under this equivalence the induced cost function on a Hessian manifold only depends on the data $(M,L)$. 
\end{remark}

Now, $\phi \mapsto f = \phi - \phi_0 $ defines a map from the space of continuous sections of $L$ to the space of continuous functions on $M$. 
\begin{lemma}
\label{lemma:cleg}
Let 
$$ f = \phi - \phi_0 $$
Then 
$$ f^c = \phi^*-\phi_0^* $$
Moreover, $d^c f$ is defined if and only if $d\phi$ is defined and $d^c f(x) = d\phi(x)$ for all $x$ where they are defined. 
\end{lemma}
\begin{proof}
Using the first point in Lemma~\ref{lemma:legdefpairing} We have
\begin{eqnarray} 
f^c(p) & = & \sup_{x\in X} -c(x,p) - f(x) = \sup_{x\in M} [x,p]-\phi(x) - \phi_0^*(p) \nonumber \\
& = & \phi^*(p) - \phi_0^*(p) \nonumber
\end{eqnarray}
proving the first part of the lemma. For the second part note that
$$ f(x) + f^c(p) + c(x,p) = \phi(x) + \phi^*(p) - [x,p] $$
hence $f^c(p) = -c(x,p)-f(x)$ if and only if $\phi^*(p) = [x,p] - \phi^*(x)$. Combining this with the second point of Lemma~\ref{lemma:legdefpairing} proves the second part.
\end{proof}

%\begin{theorem}\label{thm:optimaltransport}
%Let $(M,L)$ be a compact Hessian manifold. Let $\mu$ and $\nu$ be probability measures on $M$ and $M^*$ respectively. %Assume $\phi$ is a smooth strictly convex section of $L$ such that 
%$$ \MA_\nu(\phi) = \mu. $$
%Then $d\phi$ is the optimal transport map determined by $M,M^*,\mu,\nu$ and the cost function induced by $(M,L)$.
%\end{theorem}
\opttransport*

\begin{proof}
Let $\phi_0$ be a convex section of $L$ and $c$ be the cost function induced by $(M,L,\phi_0)$. By the first part of Lemma~\ref{lemma:cleg}, if
$$ f = \phi - \phi_0, $$
then $F(\phi)=I_c(f)$. If $\phi$ is a minimizer of $F$, then $f$ is a minimizer of $I_c$. By the second part of Lemma~\ref{lemma:cleg} $d\phi = d^cf $ which by Theorem~\ref{thm:villani} is the optimal transport map determined by $\mu$, $\nu$ and $c$. To see that our setting satisifes the conditions in Theorem~5.10 in \cite{villani2008optimal}, note that $c$ is continuous and $M$ and $M^*$ are compact manifolds (hence Polish spaces). By compactness and continuity the integrability properties in 5.10(i) and 5.10(iii) are satisfied. Moreover, by Lemma~\ref{lemma:cleg} the $c$-gradient of $f^c$ is defined (as a single valued map) almost everywhere.  
\end{proof}

Now, when $(M,L)$ is special we may take $\mu$ and $\nu$ to be the unique parallel probability measures on $M$ and $M^*$ respectively. By Theorem~\ref{thm:intro-geom} there is a smooth, strictly convex section $\phi$ of $L$ satisfying 
$$ \MA_\nu(\phi) = \mu. $$ 
Then $\Phi_q$, for some $q\in K^*$ defines a convex exhaustion function on a convex subset of $\R^n$ and $\det(\Phi_{ij})$ is constant. By J\"orgens theorem \cite{trudinger2000bernstein} $\Phi_q$ is a quadratic form and $\Omega=\R^n$. This means $\Phi$ induces an equivivariant flat metric on $\Omega$ and hence a flat metric on $M$. We conlude that any positive affine $\R$-bundle over a special Hessian manifold $M$ induces a flat Riemannian metric on $M$. 

Further, we have
\begin{lemma}\label{lemma:mmisom}
Let $(M,L)$ be a compact special Hessian manifold. Then $M$ and $M^*$ are equivalent as affine manifolds.
\end{lemma}
\begin{proof}
Let $\mu$ and $\nu$ be the unique parallel probability measures on $M$ and $M^*$ respectively and let $\phi$ be the solution to 
$$ \MA_\nu(\phi) = \mu. $$
By J\"orgens theorem $\Phi_q$ is a quadratic form. In particular, $d\Phi_q:\Omega \rightarrow \Omega^*$ is an affine map. This means $d\phi:M\rightarrow M^*$ is affine and since it is also a diffeomorphism (by Theorem~\ref{thm:diffeomorphism}) this proves the lemma.  
\end{proof}
In the following proposition and corollary we use the isomorphism in Lemma~\ref{lemma:mmisom} to identify $M$ and $M^*$.
\begin{proposition}
Let $(M,L)$ be a compact special Hessian manifold. Let $\mu$ and $\nu$ be the unique parallel probability measures on $M$ and $M^*$, respectively, and let $\phi$ be the smooth strictly convex section of $L$ satisfying 
$$ \MA_{\nu}(\phi) = \mu. $$
Then the cost function induced by $(M,L,\phi)$ is the squared distance function induced by the flat Riemannian metric determined by $L$.
\label{prop:specialcost}
\end{proposition}
\begin{proof}
Fixing $q_0\in K^*$ and letting $q=L(p)$ as in Remark~\ref{rem:leg-bi} we get
\begin{eqnarray} 
-c(x,p) & = & [x,p] - \phi(x)-\phi^*(p) \nonumber \\
& = & \sup_{\gamma\in \Pi} \gamma.q(x) - q - \phi(x)-\phi^*(p) \nonumber \\
& = & \sup_{\gamma\in \Pi} \gamma.q(x) - \gamma.q_0(x) -\Phi_{\gamma.q_0}(x) - \Phi_{q_0}^*(p) \nonumber \\
& = & \sup_{\gamma\in \Pi} -\Phi_q(\gamma^{-1}(x)) + \langle \gamma^{-1}(x),p\rangle - \Phi_{q_0}^*(p) \nonumber
\end{eqnarray}

By J\"orgens theorem \cite{trudinger2000bernstein} $\Omega=\R^n$ and for a suitable choice of $q_0\in \Gamma(X,L)$ we have 
$$ \Phi_{q_0}(x) = x^T \frac{Q}{2} x $$
for some symmetric real $n\times n$ matrix $Q$. This means $\Omega^*=(\R^n)^*$ and $$\Phi_{q_0}^*(p) = p \frac{Q^{-1}}{2} p^T. $$
Under the identification $p=d\Phi_q(x_2) = x_2^TQ$ we get
\begin{eqnarray} 
-\Phi_q(x) + \langle x,p\rangle - \Phi_{q_0}^*(p) & = & -x_1^T\frac{Q}{2}x_1 + px_1 - y\frac{Q^{-1}}{2}p^T \nonumber \\
& = & x_1^T\frac{Q}{2}x_1 - x_2^TQx_1 + x_2^T \frac{Q}{2} x_2 \nonumber \\
& = & -(x_1-x_2)^T \frac{Q}{2} (x_1-x_2). \nonumber
\end{eqnarray}
In other words 
$$ c(x,p) = -\sup_{\gamma\in \Pi} - (\gamma(x_1)-x_2) \frac{Q}{2} (\gamma(x_1)-x_2) = -d(x,p)^2/2. $$
This proves the proposition.
\end{proof}
%\begin{theorem}
%Let $(M,L)$ be a compact special Hessian manifold, $\mu$ and $\nu$ probability measures on $M$ and $M^*$ respectively. Then equation~\eqref{eq:numaeq} is equivalent to the optimal transport problem determined by $\mu$, $\nu$ and $d^2/2$, where $d$ is the flat Riemannian metric on $M$ induced by $L$.   
%\end{theorem}

\secondopttransport*
\begin{proof}
By Proposition \ref{prop:specialcost} $d^2/2$ is the cost 
function induced by $(M,L,\phi)$. The theorem then follows from Theorem~\ref{thm:optimaltransport}.
\end{proof}

\section{Einstein-Hessian metrics}
\label{sec:KE}
To illustrate that the use of the Legendre transform is not limited to the inhomogeneous Monge-Amp\'ere equation considered in the preceding section, we here consider an analogue of the K\"ahler-Einstein equation on complex manifolds, and give a variational proof of the existence of solutions. 

We first give some brief background on K\"ahler-Einstein metrics in the complex setting as motivation. For a K\"{a}hler manifold $(X,\omega)$, let $\omega_\varphi$ denote the form $\omega_\varphi = \omega + dd^c \varphi$, which we assume to be a K\"ahler form. We call $\omega_\varphi$ a K\"ahler-Einstein metric if the equation

\begin{equation}
    \Ric \omega_\varphi = \lambda \omega_\varphi
    \label{eq:KE-Ric}
\end{equation}
holds for some real constant $\lambda$. Taking cohomology, we see that for \eqref{eq:KE-Ric} to hold for some $\varphi$ we must have $\lambda [\omega] = c_1(X)$, where $c_1(X)$ denotes the first Chern class of $X$, and hence (by the $dd^c$-lemma) we have that $Ric \omega - \lambda \omega = dd^c f$ for some function $f:X\rightarrow \R$. One can show that \eqref{eq:KE-Ric} is then equivalent to solving the complex Monge-Amp\'ere equation

\begin{equation}
    (\omega + dd^c \varphi)^n = e^{f - \lambda \varphi} \omega^n.
    \label{eq:KE-MA-cplx}
\end{equation}

We here consider an analogue of \eqref{eq:KE-MA-cplx} in the setting of a compact Hessian manifold.

%\begin{theorem}
%    Let $(M,L,\phi_0)$ be a compact Hessian manifold, let $\nu$ be a probability measure of full support on $M^*$ and let %$\mu$ be a probability measure on $\mu$. Then the equation
%    
%    \begin{equation}
%        \MA_\nu \phi = e^{-\lambda (\phi - \phi_0)} \mu
%        \label{eq:KE-hess}
%    \end{equation}
%    %
%    has a solution as a convex section of a Hessian metric in $L$. 
%\label{thm:KE-hess}
%\end{theorem}
\thmKE*

To prove Theorem~\ref{thm:KE-hess} we will define a functional analogous to the Ding functional in complex geometry. It will be a modified version of the affine Kantorovich functional used in previous sections and solutions to \eqref{eq:KE-hess} are staionary points of this functional. Moreover, we will provide an additional proof of Theorem~\ref{thm:KE-hess} using an alternative functional, analogous to the Mabuchi functional in complex geometry.

\begin{definition}
    Fix $\mu_0 \in \mathcal{P}(M)$ and $\nu\in \mathcal{P}(M^*)$, where $M: \mathcal{P}(M) \to \R$, where $\mathcal{P}(M)$ denotes the space of probability measures on $M$.  We let $D:C^0(M,L) \to \R$ and $M: \mathcal{P}(M)\to \R$, be defined by
    
    \begin{align}
        D(\phi) &= \int_{M^*} \left( \phi^* - \phi_0^*\right) d\nu - \frac{1}{\lambda} \log \int_M e^{-\lambda(\phi-\phi_0)} \mu_0 \\
        M(\mu) &= \lambda \inf_{\phi\in C^0(M,L)} F_{\mu,\nu}(\phi) + \int_M \log \frac{\mu}{\mu_0} d\mu,
    \end{align}
    where $F_{\mu,\nu}$ denotes the affine Kantorovich functional \eqref{eq:kantorovich}.
    
\end{definition}

\begin{remark}
    Note that the term $\int_M \log \frac{\mu}{\mu_0} d\mu$ is precisely the relative entropy of $\mu$ with respect to $\mu_0$. Thus, the functional $M$ is finite only when $\mu$ has a density with respect to $\mu_0$. In the following we will denote this density by $\rho$, i.e., $\mu = \rho \mu_0$. 
\end{remark}

We proceed by analyzing the two functionals $D$ and $M$ separately in the following two subsections. The key point in both subsections is that existence of minimizers to $D$ and $M$ will follow from the compactness result of Theorem~\ref{thm:compactness}.

\subsection{The Ding functional}

\begin{lemma}
    $D$ descends to a functional on $C^0(M,L)/\mathbb{R}$.
    \label{lem:ding-descend}
\end{lemma}
\begin{proof}
    It is immediate to verify that $D$ is invariant under the action $\phi \mapsto \phi + c$ for any $c\in \R$. 
\end{proof}
Using the above lemma, in what follows we may choose a normalization of $\phi$ such that $\int_M e^{-\lambda(\phi-\phi_0)}\mu_0 = 1$.

\begin{lemma}
    $D$ is continuous as a map $C^0(M,L) \to \R$, and thus also as a map $C^0(M,L)/\R \to \R$. 
    \label{lem:ding-cont}
\end{lemma}

\begin{proof}
    Let $\phi_i \rightarrow \phi$ in the $C^0$-topology. The continuity of the first term was established in Proposition\ref{thm:monge-ampere},  continuity of the second term follows from the dominated convergence theorem.
\end{proof}

Lemmas~\ref{lem:ding-descend} and \ref{lem:ding-cont} then immediately give the following corollary.

\begin{corollary}
    $D$ has a convex minimizer.
    \label{cor:ding-min}
\end{corollary}
\begin{proof}
    The existence of a continuous minimizer follows from compactness and continuity. To show that the minimizer can be taken convex, it suffices to show that $D(\phi^{**}) \leq D(\phi)$. But the first term of $D$ is unchanged by double Legendre transform, by the equality $\phi^{***}= \phi^*$. Further, since $\phi^{**} \leq \phi$ for any section of $L$, we also have 
    $$\frac{-1}{\lambda} \log \int_M e^{-\lambda(\phi^{**}-\phi_0)} \mu_0 \leq \frac{-1}{\lambda} \log \int_M e^{-\lambda(\phi -\phi_0)} \mu_0$$
    for any $\lambda \neq 0$.
\end{proof}
To show that \eqref{eq:KE-hess} has a solution it thus suffices to show that minimizers of $D$ are characterized by \eqref{eq:KE-hess}.

\begin{proposition}
    Let $\phi$ be a convex section of $L$. Then $D: C^0(M)$ is Gateaux differentiable at $\phi$ and
    
    \begin{equation}
        dD|_\phi = -\MA_\nu \phi + e^{-\lambda(\phi-\phi_0)} \mu_0. 
    \end{equation}
    \label{prop:ding-char}
\end{proposition}

\begin{proof}
    The Gateaux differential of the first term was established in \ref{thm:weak_solutions}. Consider the perturbation of $\frac{-1}{\lambda} \log \int_M e^{-\lambda(\phi-\phi_0)} \mu_0$ by a continuous function $\phi \mapsto \phi + tv$, where we normalize $\int_M e^{-\lambda(\phi-\phi_0)} \mu_0 = 1$. Since $M$ is compact, the dominated convergence theorem gives that
    
    \begin{equation}
    \begin{split}
        \frac{d}{dt}|_{t=0} \log \int_M e^{-\lambda(\phi-\phi_0 + tv)}\mu_0 = \int_M  \frac{d}{dt}|_{t=0}  e^{-\lambda (\phi-\phi_0 + tv)} \mu_0 = \int_M v  e^{-\lambda (\phi-\phi_0)} \mu_0
    \end{split}
    \end{equation}
    and the proposition follows. 
\end{proof}

We now have the necessary ingredients to give a

\begin{proof}[Proof of Theorem~\ref{thm:KE-hess}]
    Note that the case where $\lambda = 0$ is Theorem~\ref{thm:intro-geom}. When $\lambda \neq 0$ the theorem follows from Corollary~\ref{cor:ding-min} and Proposition~\ref{prop:ding-char}.
\end{proof}
Uniqueness of solutions to \eqref{eq:KE-hess} when $\lambda < 0$ are also quite easy to show.

\begin{proposition}
    The solution to \eqref{eq:KE-hess} is unique modulo $\R$, for $\lambda < 0$. 
    \label{prop:KE-hess-unique}
\end{proposition}
\begin{proof}
    For simplicity assume that $\lambda = -1$, let $\psi_0,\psi_1$ be convex sections of $L$, and let $\psi_t = (1-t) \psi_0 + t\psi_1$ for $t\in(0,1)$. Note that $e^{\psi_t-\phi_0}\in L^p(\mu_0)$ for any $p\in [1,\infty]$, by continuity. Using H\"olders inequality we have that
    
    \begin{equation}
        \int_M e^{(1-t)(\psi_0 - \phi_0) + t (\psi_1-\phi_0)}d\mu_0 \leq \left(\int_M e^{\psi_0 - \phi_0}d\mu\right)^{1-t} \left(\int_M e^{\psi_1-\phi_0}d\mu_0 \right)^t
    \end{equation}
    and taking logarithms shows that $t\mapsto \int_M e^{(1-t)(\psi_0 - \phi_0) + t (\psi_1-\phi_0)}d\mu_0$ is convex in $t$. But, by the uniqueness part of Proposition~\ref{thm:monge-ampere} $\int_{M^*} (\psi_t^* - \phi_0^*) d\nu$ is strictly convex in $t$ unless $\psi_0 -\psi_1$ is constant, and hence uniqueness follows. 
\end{proof}

\subsection{The Mabuchi functional}

We also outline how to achieve the same results as above using the Mabuchi functional.

\begin{proposition}
    If $\nu$ has full support, then $M$ is lower semicontinuous. 
\end{proposition}
\begin{proof}
    By Theorem~\ref{thm:inv-ma-cont}, the first term is continuous. Further the lower semicontinuity of relative entropy is well-known.
\end{proof}

\begin{proposition}
    If $\nu$ has full support and $\mu$ has a density $\rho$ with respect to $\mu_0$, then $M$ is Gateaux differentiable, and
    
    \begin{equation}
        dM|_{\mu} = \lambda( \phi-\phi_0 ) + \log \rho,
    \end{equation}
    where $\phi$ is the unique solution to $\MA_\nu \phi = \mu$. 
    \label{prop:mab-gat}
\end{proposition}
\begin{proof}
    The Gateaux differential of relative entropy $d\left(\int_M \rho \log \rho d\mu_0\right) = \log \rho$ is well known.
    
    For the first term, let $\dot \mu$ be a perturbation of $\mu$, i.e., a measure such that $\int_M \dot \mu = 0$. Consider the function $f(t) = \inf_\phi F_{\mu + t\dot\mu,\nu}(\phi) := \inf_\phi F_t(\phi)$, i.e., the Mabuchi functional along the one-parameter family of measures given by $\dot \mu$, defined on some open interval around $t=0$. Then $F_t$ is convex (indeed, linear) in $t$, and since the space of convex sections modulo $\R$ is compact we have by Danskins theorem that $f$ has directional derivatives at $t=0$. In fact $f'_{\pm}(0) = \int (\phi_\pm - \phi_0 ) d\dot \mu$ where $\phi_\pm$ are some minimizers of $F_0$. But by the uniqueness part of Proposition~\ref{thm:monge-ampere}, $F_0$ has a unique convex minimizer $\phi$, and thus $f'(0) = \int (\phi -\phi_0 ) d\dot\mu$, by which the proposition follows.
\end{proof}

\begin{proof}[Proof of Theorem~\ref{thm:KE-hess}]
    First we note that since $M$ is compact, by Prokhorovs theorem $\mathcal{P}(M)$ is also compact. Hence, by lower semicontinuity and since $M(\mu)$ is not identically $\infty$ (e.g., $M(\mu_0)< \infty$), there is some minimizer $\mu$. Further we have that $\mu$ is absolutely continuous with respect to $\mu_0$, since otherwise $M(\mu) = \infty$. Thus $\mu = \rho \mu_0$, and by Proposition~\ref{prop:mab-gat} we have that $\MA_\nu \phi = \rho \mu_0$ and $-\lambda(\phi - \phi_0) = \log \rho$. Taking exponentials yields the theorem. 
\end{proof}

\section{Atomic measures and piecewise affine sections}
\label{sec:singular}
\begin{definition}
	We call a convex section $\phi: M\rightarrow L$ piecewise affine if for any compact set $K\subset \Omega$, it holds that $\Phi_p|_K$ is piecewise affine. 
\end{definition}

    Note that the above definition simply means that $\Phi$ can locally be written as the $\sup$ of finitely many affine sections $p_i \in \Gamma(\Omega,K)$. Note however that is \emph{not} a priori clear that this is equivalent to taking the $\sup$ over all deck transformations of finitely many $p_i \in L^*$, however this is essentially the content of the following theorem. 

%\begin{theorem}
%Let $\mu$ be an atomic probability measure on $M$, i.e., $\mu = \sum_{i=1}^N \lambda_i \delta_{x_i}$, and let $\nu$ be an %aboslutely continuous probabilty measure of full support on $M^*$. Then

%\begin{equation}
%	\MA_\nu \phi = \mu \Leftrightarrow \phi \, \text{is piecewise affine}
%\end{equation}

%\label{thm:pw-affine}
%\end{theorem}
\pwaffthm*

Note that the although the measure $\MA_\nu \phi$ depends on the choice of reference $\nu$, the condition that $\MA_\nu \phi = 0$ outside a finite set is in fact independent of $\nu$ as long as $\nu$ has a non-vanishing density. To see this, one can use the same identification \eqref{eq:m-a_equation} to identify the Monge-Amp\'ere measure with a measure on the cover $\Omega$. But this implies that $\MA_\nu \phi = 0$ if and only if $\det (\Phi_p)_{ij} = 0$, which is independent of $\nu$.

In the section that follows we will, by abuse of notation, use $\mu$ to denote both the measure on $M$, and its periodic lift to the universal cover $\tilde{M}$, which we identify with a fix convex domain in $\mathbb{R}^n$. We also identify $K^*$ with a fix convex domain in $(\R^n)^*\times \R$ by fixing $q_0\in K^*$, and letting $p$ correspond to $q_0 - p$. We further let $\Phi := \Phi_{q_0}$ and $\Phi_p = \Phi - p$. 

\begin{proof}
Fix an atomic measure $\mu$, and let $\phi$ be the solution to \eqref{eq:numaeq}. Further fix compact set $K \subset \Omega$, we aim to show that $\Phi|_K$ is piecewise affine on $K$. First we note that we may write

\begin{equation}
	\Phi(x) = \sup_{p\, \text{affine}, p \leq \Phi} p(x).
\end{equation}
We claim that it suffices to restrict the $\sup$ to

\begin{equation}
	\Phi(x) = \sup_{p(x_i)\leq \Phi(x_i)} p(x)
	\label{eq:sec-sup}
\end{equation}
as a $\sup$ over all points $x_i \in \supp \mu$. To see this, let $\tilde{\Phi}$ be the function defined in the right hand side of \eqref{eq:sec-sup}. We immediately have that $\tilde{\Phi} \geq \Phi$, by which we have that $\tilde{\Phi}^*\leq \Phi^*$. But for any point $x_i \in \supp \mu$, we also have the reverse inequality, i.e., $\tilde{\Phi}(x_i) \leq \Phi(x_i)$. Combining these observations yields that $F(\tilde{\Phi}) \leq F(\Phi)$, and by the uniqueness result in Proposition~\ref{thm:monge-ampere}, since $\Phi$ is a minimizer, we have that $\tilde{\Phi} = \Phi$.

Next we observe that since $\Phi$ is a continuous convex function, for any $x\in K$ the $\sup$ is attained at some $p \in L^*$ satisfying $\Phi(x) = p(x)$, by the Hahn-Banach theorem. More precisely, the $\sup$ is attained precisely when $p \in \partial \Phi(x)$. It follows that we may further restrict the $\sup$ to, for any $x\in K$,

\begin{equation}
	\Phi(x) = \sup_{p(x_i) \leq \Phi(x_i), p\in \partial \Phi(K)} p(x_i).
\end{equation}

Furthermore, the subdifferential image $\partial \Phi(K)$ is compact in $ K^*$, by \cite{hiriart2013convex}, Remark 6.2.3. Now fix $p\in L^*$, i.e., $p$ such that $\Phi_p$ exhausts $\Omega$. Then there is an open set $V_p \ni p$ such that $\inf_{q\in V_p} \Phi_q(x)$ also exhausts $\Omega$ (Lemma \ref{lem:uniform-exhaust}), and by compactness we may cover $\partial \Phi(K)$ by a finite collection $V_{p_j}$ of such open sets. It follows that the function

\begin{equation}
	f(y) = \inf_{x\in K} \inf_{p\in \partial \Phi(x)} \Phi_p(y) \geq \inf_{p \in \cup V_{p_j}} \Phi_p(y)
\end{equation}
also exhausts $\Omega$. But for any $y\in \Omega$, $p\in \partial \Phi(K)$ we have that $p(y) \leq \Phi(y) - f(y)$, and hence we may for $x\in K$ restrict the $\sup$ of \eqref{eq:sec-sup} to

\begin{equation}
	\Phi(x) = \sup_{p(x_i) \leq \Phi(x_i), p\in \partial \Phi(K), x_i \in \{f \leq 1\} } p(x),
\end{equation}
But the set $\{f\leq 1\}$ is compact, since $f$ exhausts $\Omega$, so $\{x \in \supp \mu, f\leq 1\}$ is finite, and thus $\Phi$ is piecewise affine on $K$. 
\end{proof}

We provide below a Lemma for convex functions in $\R^n$, which was used in Theorem~\ref{thm:pw-affine}

\begin{lemma}
    Let $f:\Omega \rightarrow \R$ be a convex exhaustion function from an open set $\Omega \subseteq \R^n$. Assume that $\Omega^* = \{f^*(p) < \infty\}$ is an open set. Then for any $p_0 \in \Omega^*$, there is an open set $U\ni p_0$ such that
    
    \begin{equation}
        g(x) := \inf_{U} f(x) - p(x)
    \end{equation}
    
    exhausts $\Omega$.
    \label{lem:uniform-exhaust}
\end{lemma}

\begin{proof}
    Without loss of generality, we assume that $p_0 = 0 \in \Omega^*$. Further we denote $f_p(x) = f(x)  - p(x)$. We claim that we may take $U = B_\delta(0)$, where $\delta > 0$ is small enough that $B_{2\delta}(0)$ is relatively compact in $\Omega^*$. The lemma follows if we can show that $\{g(x) \leq c\}$ is a closed bounded set and that $\{g(x) \leq c\} \cap \partial \Omega = \emptyset$ for every $c$, since then $\{g \leq x\}$ is a compact set in $\R^n$ contained in the open set $\Omega$. We thus proceed by showing the following three claims.

    \begin{enumerate}
        \item $g$ has bounded level sets.
        \item $g$ is continuous.
        \item $\{g(x) \leq c\} \cap \partial \Omega = \emptyset$.
    \end{enumerate}

\emph{Claim 1:} Fix $p\in B_\delta(0)$, and fix $x_0\in \Omega$ arbitrarily, and for simplicity we assume that $x_0 = 0 \in \Omega$. Let $x\in \Omega$ be arbitrary, and let $r$ be the linear function $r(y) = \delta \frac{\langle x, y \rangle}{\|x\|}$. Note that $q + r\in B_{2\delta}(0)$, and hence by assumption $f^*(q+r) \leq C_\delta$ for some constant $C_\delta$ depending only on $\delta$, by continuity of $f^*$. We then have
\begin{equation}
\begin{split}
    f_q(x) &= f_{q+r}(x) + r(x) = f_{q+r}(x) + \delta \|x\| \\
            &\geq \delta \|x\| + \inf f_{q+r}(y) = \delta \|x\| - f^*(q+r) \geq \delta \|x\| - C_\delta.
\end{split}
\end{equation}
and the first claim follows.

\emph{Claim 2:} Let $C$ be the closure of $B_\delta(0)$, and note that $g(x) = \min_{p\in C} f_p(x)$ by continuity and compactness. Thus $g(x) = f_{p_x}(x)$ for some $p_x \in C$. Fix $x\in \Omega$, and let $x_i$ be any sequence converging $x_i \rightarrow x$. Then $g(x_i) = f_{p_i}(x_i)$ for some $p_i \in C$, and by compactness we have that $p_i \rightarrow p$ for some $p\in C$, up to subsequence. But then we also have that

\begin{equation}
    g(x) \leq f_p(x) = \lim f_{p_i}(x_i) = \lim g(x_i),
\end{equation}
 by continuity of $f$ in ($p,x)$. Hence $g$ is lower semicontinuous in $x$. Upper semicontinuity follows from the fact that $g$ is defined as an $\inf$ of a family of convex functions.

\emph{Claim 3:} The first two claims show that $\{g(x) \leq c\}$ is a compact set in $\R^n$. To show the third claim it suffices to show that $g(x_i) \rightarrow \infty$ for any sequence such that $x_i \rightarrow x\in \partial \Omega$, where we may assume that $x_i$ is bounded. But for any $p\in B_\delta(0)$ we have that $f_p(x_i) = f(x_i) - p(x_i) \geq f(x_i) - \delta \|x_i\| \geq f(x_i) - C_\delta$, since $\|x_i\|$ is bounded. It follows that $g(x_i) \rightarrow \infty$.

\end{proof}

%\begin{corollary}
	%Any Hessian metric $\phi_0$ on a compact Hessian manifold can approximated uniformly by a piecewise affine potential.
%\end{corollary}
\pwaffcor*

\begin{proof}
We fix the reference measure $\nu = vol(\phi_0^*)$ as the riemannian volume form corresponding to $\phi_0^*$, and approximate $\mu$ by atomic measures $\mu_i\rightarrow \mu$. Then the solutions to $\MA_\nu \phi_i = \mu_i$ are piecewise affine, and by Theorem \ref{thm:inv-ma-cont} we have that $\phi \rightarrow \phi$ uniformly.
\end{proof}

We also note a geometric consequence of Theorem \ref{thm:pw-affine}, in that any piecewise affine convex function $\Phi: \Omega \to \mathbb{R}^n$ corresponds to a tiling of $\Omega$ by convex polytopes. Hence solving to $\MA_\nu \phi = \mu$ for an atomic measure $\mu$ yields a quasi-periodic tiling of $\Omega$.

\section{Orbifolds}
\label{sec:orbifold}
In this section we present an outline of a generalization of the main results to the setting of orbifolds. Throughout this section the setup is that of a compact affine manifold $(M)$ and the properly discontinuous affine action by a finite group $G$ on $M$. We let $X=M/G$ as a Hausdorff topological space, but since the group action $G$ is not assumed to be free $X$ is not in general a manifold. 

We call $X = M/G$ a Hessian orbifold if $M$ comes equipped with a $G$-equivariant Hessian metric $\phi$ on $M$. Note that given such a metric, the affine $\R$-bundle $L\to M$ yields a principal $\R$-bundle $L/G \to M/G$, and we denote a Hessian orbifold by the data $(M,L,G)$.

Note that sections of $L/G\to M/G$ are simply $G$-equivariant sections of $L\to M$. By letting $(M^*,L^*)$ denote the dual manifold of $(M,L)$, we as in the manifold setting may define a dual action of $G$ on $(M^*,L^*)$ and can in precisely the same way as in the manifold setting construct a dual compact Hessian orbifold $(M^*,L^*,G^*)$. Note that $G^* = G$ as groups, however we use a superscripted $*$ to indicate that $G$ acts differently on $(M^*,L^*)$. 

The extension of Theorem \ref{thm:intro-geom} can be formulated as follows

\begin{theorem}
    Let $\mu$,$\nu$ be probability measures on the compact Hessian orbifolds $(M,L,G),(M^*,L^*,G^*)$, respectively. Then the equation
    \begin{equation}
        \MA_\nu \phi = \mu
        \label{eq:MA-orb}
    \end{equation}
    has a solution. Equivalently, given $G,G^*$-invariant probability measures $\mu,\nu$ on $M,M^*$, respectively, there is a $G$-equivariant solution $\phi: M\to L$
    \begin{equation}
        \MA_\nu \phi = \mu
        \label{eq:MA-orb-upstairs}
    \end{equation}
\end{theorem}
The technique used to prove the above theorem follows the same principle as that in the manifold setting. Instead of producing solutions to \eqref{eq:MA-orb} directly, one may look for equivariant solutions to a Monge-Ampére equation on the covering manifolds $M,M^*$. The key point to note is the correspondence in the manifold setting of Hessian metrics on $M$ with equivariant convex exhaustion functions on the universal cover $\Omega$, and the extension to the orbifold setting can be seen as also requiring equivariance with respect to $G$. However, a subtle point is that to guarantee that the Kantorovich functional is somewhere finite, it is crucial that we are integrating against finite measures on $M,M^*$, whereas the corresponding measures on $\Omega,\Omega^*$ are only locally finite. In the manifold setting this correspondence between locally finite measures and probability measures is given by pushing forward under the local homeomorphisms given by the covering map $\Omega \to M$. However in the orbifold setting the quotient map $M \to M/G$ is not a covering map, and does not give local homeomorphisms near the fixed points of the action of $G$. Thus, there seems to be no obvious way to construct a probability measure on $X$ given a locally finite measure on $\Omega$. This lack of correspondence for probabilty measures is also the reason why we are not capable of dealing with non-finite groups $G$. 

Anyway, since we make the assumption that $X = M/G$ is the quotient of a compact Hessian manifold by a finite, we may push forward any probability measure on $M$ by the quotient map to yield a probability measure on $X$ (and similarly on $M^*$), and any probability measure on $X$ arises in this way.  Hence, given two $G$-equivariant measures $\mu,\nu$ on $M,M^*$, pushing forward to probability measures $\mu_X,\nu_X$ yields a Kantorovich functional

\begin{equation}
    F(\phi) = \int_X (\phi - \phi_0) d\mu_X + \int_{X^*} (\phi^* - \phi_0^*) d\nu_X
\end{equation}
The arguments in the manifold setting can then be repeated, mutatis mutandis, to yield the existence of a convex minimizer to $F$, corresponding to a convex minimizer of the Kantorovich functional on $M,M^*$ under the constraint of $G$-equivariance.

\bibliographystyle{plain}
\bibliography{legendre}

\end{document}